\documentclass[12pt]{article}

\usepackage{graphicx}
\usepackage{latexsym,amssymb}
\usepackage{amsthm}
\usepackage{indentfirst}
\usepackage{amsmath}
\usepackage{mathrsfs}
\usepackage{enumerate}
\usepackage{color}
\usepackage{xcolor}
\usepackage{hyperref}
\usepackage{fourier}
\usepackage{verbatim}
\usepackage[all]{xy}

\newcommand\SY[1]{{\noindent\color{red}  #1}}
\newcommand\sy[1]{\marginpar{\tiny \SY{#1}}} 
\newcommand{\dbar}{d\hspace*{-0.10em}\bar{}\hspace*{0.14em}}
\usepackage[a4paper,text={160true mm,230true mm},centering,top=35true mm,head=5true mm,headsep=2.5true mm,foot=8.5true mm]{geometry}

\newtheorem{theorem}{Theorem}[section]
\newtheorem{definition}{Definition}[section]
\newtheorem{proposition}[theorem]{Proposition}
\newtheorem{lemma}[theorem]{Lemma}
\newtheorem{corollary}[theorem]{Corollary}

\newtheorem{claim}[theorem]{Claim}
\newtheorem{remark}[theorem]{Remark}

\newtheorem{theoremalph}{Theorem}

\renewcommand\thetheoremalph{\Alph{theoremalph}}
\renewcommand\thecorollaryalph{\Alph{corollaryalph}}

\def\AA{{\mathbb A}} \def\BB{{\mathbb B}} \def\CC{{\mathbb C}}
\def\DD{{\mathbb D}}
\def\EE{{\mathbb E}} \def\FF{{\mathbb F}} \def\GG{{\mathbb G}}
\def\HH{{\mathbb H}}
\def\II{{\mathbb I}} \def\JJ{{\mathbb J}} \def\KK{{\mathbb K}}
\def\LL{{\mathbb L}}
\def\MM{{\mathbb M}} \def\NN{{\mathbb N}} \def\OO{{\mathbb O}}
\def\PP{{\mathbb P}}
\def\QQ{{\mathbb Q}} \def\RR{{\mathbb R}} \def\SS{{\mathbb S}}
\def\TT{{\mathbb T}}
\def\UU{{\mathbb U}} \def\VV{{\mathbb V}} \def\WW{{\mathbb W}}
\def\XX{{\mathbb X}}
\def\YY{{\mathbb Y}} \def\ZZ{{\mathbb Z}}

\def\Si{\Sigma}
\def\La{\Lambda}
\def\la{\lambda}
\def\De{\Delta}
\def\de{\delta}
\def\Om{\Omega}
\def\Ga{\Gamma}
\def\Th{\Theta}

\def\Diff{\hbox{Diff} }

\def\cA{{\cal A}}  \def\G{{\cal G}} \def\cM{{\cal M}} \def\cS{{\cal S}}
\def\cB{{\cal B}}  \def\cH{{\cal H}} \def\cN{{\cal N}} \def\cT{{\cal T}}
\def\cC{{\cal C}}  \def\cI{{\cal I}} \def\cO{{\cal O}} \def\U{{\cal U}}
\def\cD{{\cal D}}  \def\cJ{{\cal J}} \def\cP{{\cal P}} \def\V{{\cal V}}
\def\cE{{\cal E}}  \def\K{{\cal K}} \def\cQ{{\cal Q}} \def\cW{{\cal W}}
\def\F{{\cal F}}  \def\cL{{\cal L}} \def\cR{{\cal R}} \def\cX{{\cal X}}
\def\cY{{\cal Y}}  \def\cZ{{\cal Z}}

\def\vdashv{\vdash\!\dashv}
\def\triangleq{\stackrel{\triangle}{=}}
\newcommand{\id}{\operatorname{Id}}

\newcommand{\loc}{\operatorname{loc}}
\newcommand{\BBone}{{1\!\!1}}

\newcommand{\emb}{\operatorname{Emb}}
\newcommand{\sing}{{\operatorname{Sing}}}
\newcommand{\diff}{{\operatorname{Diff}}}
\newcommand{\orb}{\operatorname{Orb}}
\def\diff{\operatorname{Diff}}
\def\Int{\operatorname{Int}}
\def\dim{\operatorname{dim}}
\def\ind{\operatorname{Ind}}
\def\Per{\operatorname{Per}}
\def\diam{\operatorname{Diam}}
\def\Sing{\operatorname{Sing}}
\def\orb{\operatorname{Orb}}
\def\supp{\operatorname{Supp}}
\def\ud{\operatorname{d}}
\def\e{{\varepsilon}}
\def\CR{{\mathrm {CR}}}
\def\Interior{\operatorname{Interior}}
\def\lip{\operatorname{Lip}}
\def\Det{\operatorname{det}}
\def\Jac{\operatorname{Jac}}
\def\D{\operatorname{D}}
\def\Leb{\operatorname{Leb}}
\def\homeo{\operatorname{Homeo}}
\def\Basin{\operatorname{Basin}}
\def\Cl{\operatorname{Closure}}
\def\wh{\widehat}

\begin{document}
\begin{sloppypar}
\title{Variety of physical measures in partially hyperbolic systems with multi 1-D centers}

\author{Zeya Mi and Hangyue Zhang\footnote{H. Zhang is the corresponding author. We are partially supported by National Key R\&D Program of China(2022YFA1007800) and NSFC 12271260.
}}

\date{}
\maketitle

\begin{abstract} 
We provide two robust examples of globally partially hyperbolic systems with a multi one-dimensional center splitting, for which all Gibbs $u$-states are hyperbolic and the number of physical (Sinai–Ruelle–Bowen) measures is fixed. In the second example, the physical measures have different unstable indices, and the supports of those with index at least 2 can be adjusted to lie on a non-uniformly hyperbolic set using the techniques developed for the first example. Such partially hyperbolic systems were previously studied in \cite{MC25, Bur25}, where the existence and finiteness of physical measures were established under the assumption of hyperbolicity of the Gibbs $u$-states or the SRB measures, respectively.
\end{abstract}

\section{Introduction}\label{SEC:1}

Let $M$ be a compact Riemannian manifold and $f: M\to M$ a $C^{1+}$ diffeomorphism, meaning that $f$ is $C^1$ and $Df$ is H{\"o}lder continuous. 
An $f$-invariant measure $\mu$ is said to be a \emph{Sinai-Ruelle-Bowen(SRB)} measure if it has positive Lyapunov exponents $\mu$-almost everywhere and its conditional measures(see Rokhlin \cite{R62}) along Pesin unstable manifolds are absolutely continuous w.r.t. Lebesgue measures. Equivalently, it satisfies the Pesin entropy formula, i.e., its entropy coincides with the sum of positive Lyapunov exponents \cite{ly}. Among them, the hyperbolic ergodic SRB measures are physically observable, as they can capture the statistical behavior of points with positive Lebesgue measure. More precisely, an $f$-invariant measure $\mu$ is called a \emph{physical measure} if its \emph{basin}
$$
B(\mu,f)=\Big\{x\in M: \lim_{n\to +\infty}\frac{1}{n}\sum_{i=0}^{n-1}\varphi(f^i(x))=\int \varphi {\rm d}\mu,\quad \forall\, \varphi\in C(M)\Big\}
$$
exhibits positive Lebesgue measure.  

SRB/physical measures were firstly introduced and studied by Sinai, Ruelle and Bowen for uniformly hyperbolic systems \cite{Sin72, Rue76, Bow75, BoR75}.
A fundamental problem in smooth dynamics is to determine which systems exhibit these measures as in hyperbolic systems, see e.g. \cite{pa0,pa,v2} for related conjectures. 
Beyond uniform hyperbolicity, the existence and finiteness of SRB/physical measures have been resolved for various partially hyperbolic systems in recent decades, and 
numerous specific examples are created, including partially hyperbolic systems with non-uniformly expanding center \cite{ABV00,AD}, mostly contracting diffeomorphisms \cite{car93,Kan,bv,dvy16}, mostly expanding diffeomorphisms \cite{AV15,yjg19}.

SRB/physical measures were also investigated for a typical kind of partially hyperbolic systems--whose center bundle can be decomposed into one-dimensional sub-bundles in a dominated way.
For such systems, the SRB measures were established in \cite{CY05, CMY22}.  Under the hyperbolic hypothesis of Gibbs $u$-states(see Definition \ref{u-gibbs}) or SRB measures \cite{HYY19, MC25, Bur25}, it is shown that such system admits finitely many physical measures whose basins cover a full Lebesgue measure subset of the ambient manifold. The main task of this paper is to create non-trivial examples satisfying the assumptions in \cite{MC25, Bur25}.

Denote by ${\rm Diff}^{r}(M)$ the space of $C^r$ diffeomorphism on $M$, for $r\ge 1$. Here we produce examples consisting of diffeomorphisms on tours $\TT^n$, $n\ge 2$.
Let us present our first class of examples, each of them possessing a unique physical measure supported on some partially hyperbolic invariant set.

\begin{theoremalph}\label{mainone}
There exists $f_0\in {\rm Diff}^{\infty}(\mathbb{T}^5)$ and a $C^1$ open neighborhood $\U$ of $f_0$ such that any $f\in\U$ admits a partially hyperbolic splitting
$$
T(\mathbb{T}^5)=E^{u}\oplus_{\succ} E_1^c\oplus_{\succ} E_2^c\oplus_{\succ} E^{s}
$$
with the following properties:
\begin{enumerate}
\item \label{dimension} ${\rm dim}E_{1}^c={\rm dim}E_{2}^c=1$;
\smallskip
\item \label{nonun} neither $E_{1}^c$ nor $E_{2}^c$ is uniformly contracting/expanding;
\smallskip
\item \label{hyperbolic} any Gibbs $u$-state of $f$ is hyperbolic.
\smallskip
\end{enumerate}

Moreover, any diffeomorphism in $\U\cap {\rm Diff}^{1+}(\mathbb{T}^5)$ admits a unique ergodic physical measure, whose Lyapunov exponents along $E_1^c$ are all positive and along $E_2^c$ are all negative, and its basin covers a full Lebesgue measure subset of $\TT^5$.
\end{theoremalph}


We also obtain a family of partially hyperbolic systems having various different type physical measures, we state them as follows:

\begin{theoremalph}\label{maintwo}
Given $m\ge 2$, there exists $f_0\in {\rm Diff}^{\infty}(\mathbb{T}^{2m+1})$ and a $C^1$ open neighborhood $\U$ of $f_0$ such that
any $f\in \U$ exhibits a partially hyperbolic splitting
$$
T(\mathbb{T}^5)=E^{u}\oplus_{\succ} E_1^c\oplus_{\succ}\cdots \oplus_{\succ} E_m^c\oplus_{\succ} E^{s}, \quad {\rm dim}E_i^c=1,\quad 1\le i \le m,
$$
for which any Gibbs $u$-state is hyperbolic.

Moreover, any diffeomorphism in $\U\cap {\rm Diff}^{1+}(\mathbb{T}^{2m+1})$ admits exactly $m$ physical measures with distinct unstable indices\footnote{Here the \emph{unstable index} of an ergodic measure $\mu$ means the dimension of the unstable sub-bundle along which all the Lyapunov exponents are positive w.r.t $\mu$.}, whose union of basins cover a full Lebesgue measure subset of $\TT^{2m+1}$.
\end{theoremalph}

These two distinct types of examples (Theorems~\ref{mainone} and \ref{maintwo}) arise from perturbations of skew products. Here the base dynamics are given by a Morse–Smale system on the circle, while the fiber dynamics are linear Anosov diffeomorphisms on torous. We develop the surgery procedure originally introduced in Smale’s construction of DA-diffeomorphisms \cite{Smale} to modify the unstable indices of certain hyperbolic fixed points. As the expansion along the center direction transitions into contraction in different regions, this leads to non-uniform behavior along the center and gives rise to different types of physical measures—either supported on non-uniformly hyperbolic sets (as in Theorem~\ref{mainone}) or on uniformly hyperbolic sets with distinct unstable indices (as in Theorem~\ref{maintwo}). This phenomenon results in markedly different distributions of trajectories for Lebesgue almost every point.



Several other examples for partially hyperbolic systems having hyperbolic Gibbs $u$-states were constructed in recent years, see e.g. \cite{car93,Kan,bv,VY13, dvy16,AV15,yjg16,yjg19,MC23}. 
In contrast to these works, a key difference in our examples from Theorems \ref{mainone} and \ref{maintwo} is the emergence of Gibbs u-states characterized by multiple distinct unstable indices.
In particular, our examples do not belong to the classes of mostly contracting/expanding diffeomorphisms, introduced in \cite{bv,AV15}, or their mixture type as in \cite{MCY17,MC23}. We mention that, Viana-Yang \cite{VY13} provide examples of partially hyperbolic systems with one-dimensional center and exhibit Gibbs $u$-state/physical measures having zero central Lyapunov exponents.


\medskip
This paper is organized as follows. In Section \ref{pre}, some notions and known results
will be recalled. In Section \ref{InitialA}, we shall introduce an Anosov diffeomorphism meeting specific conditions that will be employed in the construction of examples for both Theorem \ref{mainone} and Theorem \ref{maintwo}. Section \ref{exampleone} is devoted to give a detailed construction of examples provided in Theorem \ref{mainone}. To show Theorem \ref{maintwo}, in Section \ref{twop} we construct examples for the special case of $m=2$ by employing a perturbation technique different from that used in the proof of Theorem \ref{mainone}. The general case of examples in Theorem \ref{maintwo} will be constructed in Section \ref{mg} through an inductive approach, building on the special case of $m=2$. We provide a detailed proof of Proposition \ref{DAz} in Appendix \ref{app}, which is stated in Section \ref{twop}.

\section{Preliminary}\label{pre}

Throughout this section, consider $M$ as a compact Riemannian manifold without boundary. As before, ${\rm Diff}^{r}(M)$ denotes the space of $C^r$ diffeomorphism on $M$, for $r\ge 1$.

\subsection{Lyapunov exponents}
Given $f\in {\rm Diff}^{1}(M)$, $x\in M$ and $v\in T_xM\setminus \{0\}$, we define the Lyapunov exponent of $(x,v)$ as
$$
\lambda(x,v)=\lim_{n\to \pm \infty}\frac{1}{n}\log \|Df^n(x)v\|
$$
whenever the limit exists. By Oseledets theorem \cite{O68}, for every $f$-invariant measure $\mu$, $\lambda(x,v)$ is well defined for every $v\in T_xM\setminus \{0\}$ and $\mu$-almost every $x\in M$.
Let $E$ be a one-dimensional $Df$-invariant sub-bundle, denote the Lyapunov exponent of $E$ at point $x\in M$ by
$$
\chi(x,E)=\lim_{n\to +\infty}\frac{1}{n}\log \|Df^n|_{E(x)}\|
$$
when the limit exists. By the Birkhoff ergodic theorem, for every $f$-invariant measure $\mu$, $\chi(x,E)$ is well defined $\mu$-almost everywhere, and we have
$$
\chi(\mu,E):=\int \chi(x,E) {\rm d}\mu=\int \log \|Df|_{E}\|{\rm d}\mu.
$$
When $\mu$ is ergodic, then $\chi(\mu,E)= \chi(x,E)$ for $\mu$-almost every $x$.

\subsection{Partial hyperbolicity}
Let $f\in {\rm Diff}^1(M)$ and $\Lambda$ be an $f$-invariant compact subset of $M$. A $Df$-invariant splitting $T_{\Lambda}M=E\oplus_{\succ} F$ over $\Lambda$ is called a \emph{dominated splitting} if $E$ dominates $F$ in the sense that there exist constants $\lambda\in (0,1)$ and $C>0$ such that for every $x\in \Lambda$ and $n\in \NN$, one has 
$$
\|Df^n|_{F(x)}\|\cdot \|Df^{-n}|_{E(f^n(x))}\|\le C\lambda^n.
$$
We say $\Lambda$ is partially hyperbolic if there exists a dominated splitting $T_{\Lambda}M=E^u\oplus_{\succ} E_1\oplus_{\succ}\cdots \oplus_{\succ} E_k\oplus_{\succ} E^s$ for some $k\ge 1$ on the tangent bundle over $\Lambda$, for which 
\begin{itemize}
\item $E^u$ is uniformly expanding and $E^s$ is uniformly contracting,
\item at least one of $E^u$ and $E^s$ is nontrivial.
\end{itemize}

It is well known that $E^s$ and $E^u$ are uniquely integrable to $f$-invariant foliations, called strong stable and strong unstable foliations, whose leaves are called strong stable and strong unstable manifolds. 

A diffeomorphism $f\in {\rm Diff}^1(M)$ is called partially hyperbolic if $M$ is a partially hyperbolic set. In the present paper, we will focus on partially hyperbolic diffeomorphisms within multi one-dimensional centers, i.e., ${\rm dim}E^i=1$ for $1\le i \le k$. 

\subsection{Gibbs $u$-states}

As proposed by \cite{CYZ18}, we give the definition of Gibbs $u$-state via partial unstable entropy for the $C^1$ partially hyperbolic diffeomorphism exhibiting $E^u$. Denote by $h_{\mu}^u(f)$ the partial entropy along unstable foliation tangent to $E^u$(see \cite{HHW,yjg16} for more details).

\begin{definition}\label{u-gibbs}
We call an $f$-invariant measure $\mu$ a Gibbs $u$-state if it satisfies the partial unstable entropy formula:
$$
h_{\mu}^u(f)=\int \log |{\rm det}Df|_{E^u}|{\rm d}\mu.
$$
\end{definition}

According to \cite{L84}, if $f$ is $C^{1+}$, then $\mu$ is a Gibbs $u$-state if and only if its conditional measures along strong unstable manifolds are absolutely continuous w.r.t. Lebesgue measures of these manifolds.

Given $f\in {\rm Diff}^1(M)$ and $x\in M$, denote by $\mathscr{M}(x,f)$ the set of limit measures of the empirical measures $\{\frac{1}{n}\sum_{i=0}^{n-1}\delta_{f^i(x)}\}_{n\in \NN}$.
Let  $\G^u(f)$ be the space of Gibbs $u$-states of $f$. The following result will be used throughout this paper.

\begin{proposition}\label{basicu}\cite{BDV05,CYZ18,HYY19}
Let $f$ be a $C^1$ partially hyperbolic diffeomorphism exhibiting $E^u$. Then
\begin{enumerate}[(1)]
\item\label{cpu} any ergodic component of the Gibbs $u$-state is also a Gibbs $u$-state of $f$;
\item\label{upp} if $f_n$ converges to $f$ in $C^1$-topology, $\mu_n\in \mathcal{G}^u(f_n)$ and $\mu_n\to \mu$ in weak$^{\ast}$-topology, then $\mu\in \G^u(f)$;
\item \label{ufull} for Lebesgue almost every $x\in M$, $\mathscr{M}(x,f)\subset \G^u(f)$.
\end{enumerate}
\end{proposition}

%
%

\section{Initial transitive Anosov diffeomorphism}\label{InitialA}

In this section, we construct a specific Anosov diffeomorphism as the foundational step for building the examples in both Theorems \ref{mainone} and \ref{maintwo}.
Let us begin by fixing $\delta_0>0$ sufficiently small so that $\delta_0<1/10000$ and consider the truncation function $\varphi:\RR\to \RR$ with the following properties:
\begin{itemize}
    \item \( \varphi(x) = \varphi(-x) \), i.e., \( \varphi \) is symmetric about \( x = 0 \).
    \smallskip
    \item \( \varphi(x) \) is strictly monotonic for \( x \in ( \delta_0/4, \delta_0/2 ) \).
    \smallskip
    \item \( \varphi(x) = 1 \) for \( x \in \left[ 0, \delta_0/4 \right] \), and \( \varphi(x) = 0 \) for \( x \ge \delta_0/2 \).
\end{itemize}
We will make use the following two functions
$$
H(x,y)=\varphi(x)(\varphi(y)+y\cdot \varphi'(y)), 
$$
$$
R(x,y,z)=\varphi(x)\cdot \varphi(y)\cdot(\varphi(z)+z\cdot \varphi(z))
$$
defined on $\RR^2$ and $\RR^3$, respectively. 

Let $A_0$ be the linear Anosov diffeomorphism on $\TT^2=\RR^2/\ZZ^2$ induced by 
$
\begin{pmatrix}
2 & 1\\
1 & 1 
\end{pmatrix},
$
which admits two eigenvalues 
$$
\lambda_0=\frac{3+\sqrt{5}}{2},\quad \lambda_0^{-1}=\frac{3-\sqrt{5}}{2}.
$$
Note that each point admits stable and unstable manifolds with respect to \( A_0 \), which are tangent to the eigenspaces associated with the expanding and contracting eigenvalues, respectively. Let us introduce \(\mathbb{T}^2\)-isometric coordinates $(x,y)$ on an open neighborhood $U$ of the fixed point $0\in \TT^2$, such that points with coordinates $(x,0)$ are in the unstable manifold of $0$ and points with coordinates $(0,y)$ are in the stable manifold of $0$. Assume that $[-\delta_0,\delta_0]^2\subset U$.

Let us fix $N\in \NN$ large enough such that 
$$
A:=A_0^N,\quad \lambda:=\lambda_0^N
$$ 
exhibit the properties as follows:
\begin{enumerate}[(H1)]
\item \label{pou} $A$ admits a fixed point $p\in \TT^2$ outside $U$, 
\item \label{one}
$$
\lambda \ge 6,
$$
\item \label{two}
$$
\lambda^{2}\ge 3(C+1)\lambda-\frac{3}{2}C,
$$
where \(C\) is a bound of the functions $H$ and $R$. 
\item \label{three}
$$
\frac{\log \lambda}{\log 3\lambda}>  {\rm Leb}_{\TT^2}([-\delta_0,\delta_0]^2),
$$
where ${\rm Leb}_{\TT^2}$ denotes the normalized Lebesgue measure on $\TT^2$.
\end{enumerate}

The product map $B = A^{2} \times A : \TT^2\times \TT^2  \to \TT^2\times \TT^2$ admits four eigenvalues $\lambda^2$,$\lambda$,$\lambda^{-1}$,$\lambda^{-2}$, with the associated eigenspaces
${\cE}^{uu},{\cE}^u,{\cE}^{s},{\cE}^{ss}$. Thus, $B$ is a transitive Anosov diffeomorphism exhibiting the partially hyperbolic splitting 
$$
T(\TT^4)={\cE}^{uu}\oplus_{\succ} {\cE}^{u}\oplus_{\succ} {\cE}^s\oplus_{\succ} {\cE}^{ss}.
$$
There exist foliations $\F^{uu},\F^u,\F^s,{\F}^{ss}$ tangent to ${\cE}^{uu}, {\cE}^{u}, {\cE}^{s},{\cE}^{ss}$, respectively. Denote by $\F^{\ast}(x)$ the leaf of $\F^{\ast}$ containing $x\in \TT^4$, for every $\ast\in \{uu,u,s,ss\}$.

There exists an isometric coordinate chart $V\subset U\times U$ on \( \mathbb{T}^4 \), given by 
\[
 \mathcal{F}^{uu}_{\delta_0}(0) \times \mathcal{F}^u_{\delta_0}(0) \times \mathcal{F}^s_{\delta_0}(0) \times \mathcal{F}^{ss}_{\delta_0}(0),
\]
with each \( \mathcal{F}^\ast_{\delta_0}(0) \) denoting the local plaque of size \( \delta_0 \) of the corresponding (un)stable foliation, centered at the fixed point \( 0 \). For convenience, we identify \( V \) with the cube \( [-\delta_0, \delta_0]^4 \).  
Under this identification, we can assume that
$$
B(x, y, z, w) = (\lambda^2 x, \lambda y, \lambda^{-1} z, \lambda^{-2} w)   \quad \textrm{if}~(x, y,z,w) \in [-\delta_0, \delta_0]^4.
$$
For the sake of clarity, we can regard this as an action lifted to \( \mathbb{R}^4 \).

\section{Diffeomorphisms within a unique physical measure}\label{exampleone}

\subsection{Construction of $\{F_k\}$}

Let $B = A^{2} \times A : \TT^2\times \TT^2  \to \TT^2\times \TT^2$ be the the linear Anosov diffeomorphism given in Section \ref{InitialA}. 
Recall that $\lambda$ is a constant satisfying (H\ref{one})-(H\ref{three}). Let $K$ be a $C^{\infty}$ Morse-Smale diffeomorphism on $\mathbb{S}^1=[0,1]/\{0,1\}$, whose non-wandering set is consists of a sink $a_1=0$ and a source $a_2=1/2$, for which
\begin{enumerate}[(K1)]
\item\label{Kone}
$$
\lambda^{-1}<\|DK(x)\|<\lambda,\quad x\in \SS^1,
$$
\item\label{Ktwo}
$$
\|DK(x)\|<\frac{3}{2}\lambda^{-1},\quad x\in [-\delta_0,\delta_0],
$$
\item
$$
\|DK(x)\|>\frac{2}{3}\lambda,\quad x\in [1/2-\delta_0,1/2+\delta_0].
$$
\end{enumerate}
%
%
%
We then obtain the diffeomorphism $F=B\times K$ on $\TT^5$, which admits the partially hyperbolic splitting 
$$
T(\TT^5)={\cE}^{uu}\oplus_{\succ} {\cE}^{u}\oplus_{\succ} T\mathbb{S}^1 \oplus_{\succ} {\cE}^{s} \oplus_{\succ} {\cE}^{ss}.
$$
%
%

For each fixed point $z\in \TT^5$, $F=B\times K$ acts locally on the product of the manifolds $\F^{uu}(z)\times \F^{u}(z)\times \SS^1(z)\times \F^s(z)\times \F^{ss}(z)$ according to the following rule:
$$
F(x_1,x_2,x_3,x_4,x_5)=(\lambda^2 x_1,\lambda x_2, K(x_3), \lambda^{-1} x_4,\lambda^{-2} x_5),
$$
where \(x = (x_1, x_2, x_3, x_4, x_5) \in \mathcal{F}^{uu}(z) \times \mathcal{F}^u(z) \times \SS^1(z) \times \mathcal{F}^s(z) \times \mathcal{F}^{ss}(z)\), and \(\SS^1(z)\) is the circle leaf containing the point \(z\), which is tangent to \(T\SS^1\) at every point.

Setting
$$
U_0= [-\delta_0, \delta_0]^5=\F^{uu}_{\delta_0}(0)\times \F^{u}_{\delta_0}(0)\times \SS^1_{\delta_0}(0) \times \F^s_{\delta_0}(0)\times \F_{\delta_0}^{ss}(0),
$$
which is a neighborhood of the fixed point $0$ in $\TT^5$. 
Now we construct a family of diffeomorphisms $\{F_k\}_{k\in \NN}$ by deforming $F$ inside $U_0$. 
For $k\in \NN$, let $F_k$ be defined as follows:
\begin{itemize}
\smallskip
\item $F_k$ agrees with $F$ for points outside $U_0$, 
\smallskip
\item if $x=(x_1,x_2,x_3,x_4,x_5)\in U_0$,
put
$$
F_k(x)=\left(\lambda^2 x_1,P_k(x),K(x_3),\lambda^{-1}x_4,\lambda^{-2}x_5\right),
$$
where 
$$
P_k(x) = \varphi(kx_2)\cdot \varphi\left(\sqrt{\sum_{j=1,j\neq 2}^5x_j^2}\right) \cdot \left( \frac{1}{2}- \lambda\right)x_2 + \lambda x_2.
$$
\end{itemize}

\medskip
To verify that $F_k$ is a $C^{\infty}$ diffeomorphism, we need the following result.

\begin{lemma}\label{splitting}
It holds that
\begin{equation}\label{peb}
\frac{\partial}{\partial x_2}P_k(x)\in \left[\frac{1}{2}, \frac{\lambda^2}{3}\right],~ \forall~x\in U_0,\quad\frac{\partial}{\partial x_2}P_k(0)=\frac{1}{2};
\end{equation}
\begin{equation}\label{foru} 
\frac{\partial }{\partial x_i}P_k(x)=0~ \textrm{if}~x_i=0, \quad
\lim_{k\to \infty}\sup_{x\in U_0}\left\{\frac{\partial }{\partial x_i}P_k(x)\right\}=0,\quad i\neq 2;
\end{equation}
\end{lemma}

\begin{proof}
Given $x=(x_1,\cdots,x_5)\in U_0$, we have
\begin{eqnarray*}
\frac{\partial }{\partial x_2}P_k(x) &=& \varphi\left(\sqrt{\sum_{j=1,j\neq 2}^5x_j^2}\right)\left(\frac{1}{2} - \lambda\right) \left(\varphi(kx_2) + k x_2 \cdot \varphi'(kx_2) \right) + \lambda\\
&=& H\left(\sqrt{\sum_{j=1,j\neq 2}^5x_j^2}, kx_2\right)\left(\frac{1}{2}-\lambda\right)+\lambda.
\end{eqnarray*}
We see directly that
$$
\frac{\partial}{\partial x_2}P_k(0)=\frac{1}{2}.
$$
Observe that
$$
0\le \varphi(p)\le 1, \quad p\cdot\varphi'(p)\le 0, \quad \forall ~p\in \RR,
$$
and $-C \le H(x,y) \le 1$ whenever $(x,y)\in \RR^2$.
It follows that
$$
-C\left(\frac{1}{2} - \lambda\right) + \lambda\ge\frac{\partial }{\partial x_2}P_k(x)\ge \left(\frac{1}{2}-\lambda\right)+\lambda=\frac{1}{2}.
$$
Consequently, by the choices of $\lambda$,  
$$
\frac{\partial }{\partial x_2}P_k(x)\le -C\left(\frac{1}{2} - \lambda\right) + \lambda\le \frac{\lambda^2}{3}.
$$
Thus, we get (\ref{peb}).

\medskip
Now we show \eqref{foru}. A simple computation yields
$$
\frac{\partial}{\partial x_i}P_k(0)=0, \quad i\neq 2.
$$
For every $x=(x_1,\cdots,x_5)\in U_0\setminus\{0\}$, we have 
$$
\displaystyle
\frac{\partial }{\partial x_i}P_k(x)= \varphi(kx_2) \cdot \left( \frac{1}{2} - \lambda \right) x_2 \cdot \varphi'\left( \sqrt{\sum_{j=1,j\neq 2}^5x_j^2} \right) \cdot \frac{x_i}{\sqrt{\sum_{j=1,j\neq 2}^5x_j^2}},\quad i\neq 2.
$$
We thus obtain the first part of \eqref{foru} by definition. Moreover,
\begin{eqnarray*}
\left|\frac{\partial }{\partial x_i}P_k(x)\right|&\le& |\varphi(kx_2)\cdot x_2|\cdot \sup_{p\in \RR}|\varphi'(p)|\cdot (\lambda-1/2)\\
&\le & \delta/2k\cdot \sup_{p\in \RR}|\varphi'(p)|\cdot (\lambda-1/2),
\end{eqnarray*}
where we use the fact $|\varphi(kp) \cdot p| \leq \delta/2k$ for every $p\in\RR$. Since $\sup_{p\in \RR}|\varphi'(p)|<+\infty$, one gets the convergence in \eqref{foru}.


\end{proof}

%
%

\begin{lemma}\label{Cinfinite}
Every $F_k$ is a $C^{\infty}$ diffeomorphism on $\TT^5$.
\end{lemma}

\begin{proof}
For each $k\in \NN$, consider the map $I_k$ defined as follows: it coincides with the identity map $I$ outside $U_0$, and for every $x=(x_1, \cdots , x_5) \in U_0$, 
$$
I_k(x) = \left(x_1, \frac{P_k(x)}{\lambda}, x_3, x_4, x_5 \right). 
$$
As a first step, we show that \( I_k \) is a \( C^\infty \) diffeomorphism.
Since we have 
$$
\frac{\partial }{\partial x_2} P_k(x)\geq \frac{1}{2}>0,\quad \forall~ x\in U_0,
$$ 
by \eqref{peb} in Lemma~\ref{splitting}, which implies that the function $P_k$ is strictly monotonic for the second variable $x_2$ on the interval \( [-\delta_0, \delta_0] \). Furthermore, one can check directly that
$$
\lambda^{-1}\cdot P_k\left(x_1,\pm\frac{2}{3}\delta_0,x_3,x_4,x_5\right)=\pm\frac{2}{3}\delta_0.
$$
We observe that $\lambda^{-1}\cdot P_k$ is the identity map for points outside $[-\frac{2}{3}\delta_0, \frac{2}{3}\delta_0]^5$. As a conseqence, the map $I_k$ is a homeomorphism on $\TT^5$. Then, use the \( C^\infty \) of the truncation \( \varphi \), we know \( I_{k} \) is a \( C^\infty \) diffeomorphism.  Notice that $F_k=F\circ I_k$ (recall $F=B\times K$). Since \( F \) is \( C^\infty \), and \( I_k \) is a \( C^\infty \) diffeomorphism, it follows that \( F_k \) is also a \( C^\infty \) diffeomorphism.
This completes the proof.
\end{proof}


\medskip
Denote by $\pi_{\SS^1}$ the natural projection map from $\TT^4\times \SS^1$ onto $\SS^1$. In terms of the splitting ${\cE}^{uu}\oplus {\cE}^{u}\oplus T\mathbb{S}^1 \oplus {\cE}^{s} \oplus {\cE}^{ss}$ on the tangent space $T(\TT^5)$, for every $k\in \NN$, one can check by definition that: for every $x\in \TT^5$, 
\begin{equation}\label{exp1}
DF_k(x)=
\renewcommand{\arraystretch}{1.3}
\begin{pmatrix}
\lambda^2 & 0 & 0 & 0 & 0 \\
\xi_1^k(x) & \xi_2^k(x)  & \xi_3^k(x) & \xi_4^k(x) & \xi_5^k(x) \\
0 & 0 & DK(\pi_{\SS^1}(x)) & 0 & 0 \\
0 & 0 & 0 & \lambda^{-1} & 0 \\
0 & 0 & 0 & 0 & \lambda^{-2}
\end{pmatrix}
,
\end{equation}
where 
$$
\displaystyle
\left\{
\begin{array}{ll}
  \xi_i^k(x) = \dfrac{\partial}{\partial x_i} P_k(x), \quad 1 \le i \le 5, 
    & ~~x \in U_0~; \\[2ex]
  \xi_2^k(x) = \lambda, \quad \xi_i^k(x) = 0, \quad i \neq 2, 
    & ~~x \notin U_0~.
\end{array}
\right.
$$

\subsection{Existence of the partially hyperbolic splitting}

Let $E, F$ be two sub-bundles of the tangent space $T(\TT^5)$,  where the intersection of \( E \) and \( F \) is a trivial bundle. Given $\alpha>0$, we define the \emph{cone field} $\mathscr{C}_{\alpha}(E,F)=\{\mathscr{C}_{\alpha}(E,F,x), x\in \TT^5\}$ by 
$$
\mathscr{C}_{\alpha}(E,F,x)=\left\{v=v_E+v_F\in E(x)\oplus F(x): \|v_F\|\le \alpha \|v_E\|\right\}, \quad x\in \TT^5.
$$
We have the following result.

\begin{lemma}\label{first}
For any $\varepsilon>0$, there exists $k_{\varepsilon}\in \NN$ such that for any $k\ge k_{\varepsilon}$, $F_k$ admits the partially hyperbolic splitting 
$$
T(\mathbb{T}^5)=E^{u}\oplus_{\succ} E_1^c\oplus_{\succ} E_2^c\oplus_{\succ} E^{s}
$$
such that
\begin{enumerate}[(1)]
\item $E^u\subset \mathscr{C}_{\varepsilon}({\cE}^{uu}, {\cE}^u\oplus T\SS^1)$ and is uniformly expanding;
\smallskip
\item $E^s\subset \mathscr{C}_{\varepsilon}({\cE}^{s}\oplus {\cE}^{ss}, {\cE}^{u}\oplus T\SS^1)$ and is uniformly contracting;
\item \label{E1nh}$E_1^c= {\cE}^u$, which satisfies 
$\|DF_k|_{E_1^c(x)}\|=\lambda>1$ if $x\notin U_0$ and equals to $1/2$ at $0\in \TT^5$;
\item $E_2^c\subset \mathscr{C}_{\varepsilon}(T\SS^1, {\cE}^u)$.
\end{enumerate}
\end{lemma}

\begin{proof}
For every $k\ge 1$, the expression of $DF_k$ in \eqref{exp1} implies that the following sub-bundles
$$
\cE^u,\quad \cE^{uu}\oplus \cE^u,\quad \cE^u \oplus T\SS^1\oplus \cE^s\oplus \cE^{ss},\quad \cE^u\oplus T\SS^1,
$$
are $DF_k$-invariant. Moreover, we can show the absolute values of diagonal entries of $DF_k$ are dominated one by one as follows:
\begin{equation}\label{dm}
\lambda^2>\xi_2^k(x)>\|DK(\pi_{\SS^1}(x))\|>\lambda^{-1}>\lambda^{-2},\quad \forall ~x\in \TT^5.
\end{equation}
Indeed, if $x\notin U_0$, then $\xi_2^k(x)=\lambda$, \eqref{dm} follows from (K\ref{Kone}); in the case that $x\in U_0$, using \eqref{peb} in Lemma \ref{splitting}, (H\ref{one}) and (K\ref{Ktwo}), one deduces that 
$$
\lambda^2>\xi_2^k(x)=\frac{\partial }{\partial x_2}P_k(x)>\| DK(\pi_{\SS^1}(x)) \| >\lambda^{-1},
$$ 
which yields \eqref{dm} as well.
%
By \eqref{dm}, \eqref{foru} in Lemma \ref{splitting} and take expression \eqref{exp1} into account, one deduces that for every $\varepsilon>0$, there exist $\kappa\in (0,1)$ and $k_{\varepsilon}\in \NN$ such that for every $k\ge k_{\varepsilon}$, for every $x\in \TT^5$, it holds that
$$
DF_k(x)\left(\mathscr{C}_{\varepsilon}({\cE}^{uu}, {\cE}^u,x)\right)\subset \mathscr{C}_{\varepsilon\kappa}({\cE}^{uu}, {\cE}^u,F_k(x)),
$$
$$
DF_k^{-1}(x)\mathscr{C}_{\varepsilon}({\cE}^{s}\oplus {\cE}^{ss}, {\cE}^{u}\oplus T\SS^1,x)\subset \mathscr{C}_{\varepsilon\kappa}\left({\cE}^{s}\oplus {\cE}^{ss}, {\cE}^{u}\oplus T\SS^1, F_k^{-1}(x)\right),
$$
$$
DF_k^{-1}(x)\mathscr{C}_{\varepsilon}(T\SS^1, {\cE}^u,x)\subset \mathscr{C}_{\varepsilon\kappa}\left(T\SS^1, {\cE}^u,F_k^{-1}(x)\right).
$$

Now we fix any large $k\ge k_{\varepsilon}$. Let us take $E_1^c:=\cE^u$ directly. Using the cones argument, one can find sub-bundles $E^u\subset \mathscr{C}_{\varepsilon}({\cE}^{uu}, {\cE}^u)$, $E^s\subset \mathscr{C}_{\varepsilon}({\cE}^{s}\oplus {\cE}^{ss}, {\cE}^{u}\oplus T\SS^1)$ and $E_2^c\subset \mathscr{C}_{\varepsilon}(T\SS^1, {\cE}^u)$, they are $DF_k$-invariant. Among these, \( E^u \) is uniformly expanding, and \( E^s \) is uniformly contracting.
By the expression \eqref{exp1} of $DF_k$ and the fact $E_1^c=\cE^u$, we know
$$
\|DF_k|_{E_1^c(x)}\|=\xi_2^k(x), \quad \forall~x\in \TT^5.
$$
By definition of $\xi_2^k$ and using \eqref{peb} in Lemma~\ref{splitting}, we obtain 
$$
\|DF_k|_{E_1^c(x)}\|=\lambda>1,\quad \forall~x\notin U_0, 
$$
and 
$$
\|DF_k|_{E_1^c(0)}\|=\frac{\partial}{\partial x_2}P_k(0)=\frac{1}{2}.
$$
To be summarized, $F_k$ admits a partially hyperbolic splitting
$$
T(\mathbb{T}^5)=E^{u}\oplus_{\succ} E_1^c\oplus_{\succ} E_2^c\oplus_{\succ} E^{s}
$$
with the desired properties.
\end{proof}

\subsection{Proof of Theorem \ref{mainone}}

Throughout the remainder of this section, we fix \( f_0 = F_k \) for some \( k \geq k_1 \) that satisfies Lemma \ref{first}.
We will obtain Theorem \ref{mainone} by showing the following more detailed result.

\begin{theorem}\label{gtop}
There exists a $C^1$ open neighborhood $\U$ of $f_0$ such that for every $f\in \U$, there exists a partially hyperbolic splitting
$$
T(\mathbb{T}^5)=E^{u}\oplus_{\succ} E_{1}^c\oplus_{\succ} E_{2}^c\oplus_{\succ} E^{s}
$$
such that
\begin{enumerate}[(1)]
\item \label{nonun} neither $E_{1}^c$ nor $E_{2}^c$ is uniformly contracting/expanding;
\smallskip
\item \label{dimension} ${\rm dim}E^u={\rm dim}E_{1}^c={\rm dim}E_{2}^c=1$, ${\rm dim}E^s=2$;
\smallskip
\item \label{hyperbolic} any Gibbs $u$-state of $f$ is hyperbolic;
\smallskip
\item \label{empirical} for Lebesgue almost every $x\in \mathbb{T}^5$, any measure of $\mathscr{M}(x,f)$ admits only positive Lyapunov exponents along $E_{1}^c$ and negative Lyapunov exponents along $E_{2}^c$;
\smallskip
\item \label{uniqueness} any $C^{1+}$ diffeomorphism in $\U$ admits a unique physical measure whose Lyapunov exponents along $E_1^c$ are all positive and along $E_2^c$ are all negative, and its basin covers a full Lebesgue measure subset of $\TT^5$.
\end{enumerate}
\end{theorem}

%

\subsubsection{Dynamics of diffeomorphisms nearby $f_0$}

\begin{lemma}\label{orcont}
There exists $\theta\in (0,1)$ such that 
$$
\|Df_0^{-1}|_{E_2^c(x)}\|=\|Df_0^{-1}|_{T_x\SS^1}\|<\theta,\quad |{\rm det} Df_0^{-1}(x)|<\theta, \quad \forall~x\in \TT^4\times \{a_2\}.
$$
$$
\|Df_0|_{E_2^c(x)}\|=\|Df_0|_{T_x\SS^1}\|<\theta, \quad |{\rm det} Df_0(x)|<\theta,\quad \forall~x\in \TT^4\times \{a_1\}.
$$

\end{lemma}

\begin{proof}
For points in $\TT^4\times \{a_2\}$, as in this case $f_0=B\times K$, by definition one can find $\alpha\in (0,1)$ to satisfy the associated inequalities.
We need only to show the inequalities for points of $\TT^4\times \{a_1\}$.

Noting first that 
\begin{equation}\label{sinv}
\xi_3^k(x)=0,\quad \forall~x\in \TT^4\times \{a_1\}. 
\end{equation}
Indeed, we get $\xi_3^k(x)=0$ when $x\notin U_0$ by definition. When $x\in U_0$, then it has the expression $x=(x_1,\cdots,x_5)$ with $x_3=0$. By \eqref{foru} in Lemma \ref{splitting}, we have also
$$
\xi_3^k(x)=\frac{\partial }{\partial x_3}P_k(x)=0.
$$
Altogether, we obtain \eqref{sinv}. 

From \eqref{sinv} and the expression of $Df_0=DF_k$ provided in \eqref{exp1}, we see that $T\SS^1$is invariant under $Df_0$ restricting on the fiber $\TT^4\times \{a_1\}$. Moreover, by the domination between $\cE^u$ and $T\SS^1$ and the fact $E_2^c\subset \mathscr{C}_{1}(T\SS^1, {\cE}^u)$, it follows that $E_2^c$ coincides with $T\SS^1$ on $\TT^4\times \{a_1\}$. Thus, for every $x\in \TT^4\times \{a_1\}$,
$$
\|Df_0|_{E_2^c(x)}\|=\|Df_0|_{T_x\SS^1}\|=\|DK(a_1)\|<\frac{3}{2}\lambda^{-1}<1/4.
$$
Furthermore, by definition of $\xi_2^k(x)$ and \eqref{peb} in Lemma \ref{splitting}, it follows that
$$
\|Df_0|_{E_1^c(x)}\|=\xi_2^k(x)\le \max\{\lambda,\lambda^2/3\}=\lambda^2/3,
$$
where we use the fact $\lambda>6$ in (H\ref{one}). Thus, recalling (H\ref{Ktwo}), we have
\begin{eqnarray*}
|{\rm det}Df_0(x)|&\le &\lambda^2\cdot \|Df_0|_{E_1^c(x)}\|\cdot \|DK(a_1)\|\cdot \lambda^{-1}\cdot \lambda^{-2}\\
&\le & \frac{\lambda^2}{3}\cdot \lambda^{-1}\cdot \frac{3}{2}\lambda^{-1}\\
&=& \frac{1}{2}.
\end{eqnarray*}
To be summarized, we obtain the desired result by taking $\theta=\max\{1/2,\alpha\}$. 
\end{proof}

By the $C^1$-robustness of partial hyperbolicity, we can extend estimates in Lemma \ref{orcont} to diffeomorphisms nearby $f$. More precisely, we have 


\begin{proposition}\label{topar}
There exists a $C^1$ open neighborhood $\U_1$ of $f$ and constants $\theta\in(0,1)$, $0<\eta\ll 1/4$ such that any $f\in \U_1$ exhibits the partially hyperbolic splitting
$$
T(\mathbb{T}^5)=E^{u}_f\oplus_{\succ} E_{1,f}^c\oplus_{\succ} E_{2,f}^c\oplus_{\succ} E^{s}_f,
$$
for which
\begin{equation}\label{exttonei}
\max_{x\in  \TT^4\times [a_1-\eta, a_1+\eta]}\left\{\|Df|_{T_x\SS^1}\|,\|Df|_{E^c_{2,f}(x)}\|, |{\rm det} Df(x)|\right\}<\theta,
\end{equation}
\begin{equation}\label{exctonei}
\max_{x\in  \TT^4\times [a_2-\eta, a_2+\eta]}\left\{\|Df^{-1}|_{T_x\SS^1}\|,\|Df^{-1}|_{E_{2,f}^c(x)}\|, |{\rm det} Df^{-1}(x)|\right\}<\theta.
\end{equation}
\end{proposition}

We emphasize that for general $f\in \U_1$, $E_{2,f}^c$ and $T\SS^1$ do not necessary agree, even in $\TT^4\times \{a_1, a_2\}$. The sub-bundles in the splitting converge to the associated sub-bundles of $f_0$,
as $f$ converges to $f_0$ in $C^1$-topology. 
For the rest of this paper, we will drop the subscript $f$ in these sub-bundles if there is no ambiguity.


\begin{proposition}\label{invariant}
There exists a $C^1$ open neighborhood $\U_2\subset \U_1$ of $f_0$ such that for every $f\in \U_2$, there exist an attractor $\mathcal{A}_f$ and a repeller $\mathcal{R}_f$ such that 
\begin{enumerate}[(1)]
\item \label{zerol} ${\rm Leb}(\mathcal{A}_f)={\rm Leb}(\mathcal{R}_f)=0$;
\smallskip
\item \label{orbit} if $x\notin \mathcal{A}_f\cup \mathcal{R}_f$, then
$f^n(x)\to \mathcal{A}_f$ and $f^{-n}(x)\to \mathcal{R}_f$ as $n\to +\infty$;
\smallskip
\item\label{nu2}  $\|Df|_{E_{2}^c(x)}\|\le\theta$ for every $x\in \mathcal{A}_f$, $\|Df^{-1}|_{E_{2}^c(x)}\|\le \theta$ for every $x\in \mathcal{R}_f$.
\end{enumerate}
\end{proposition}

\begin{proof}
Let $\U_1$ and $\eta$ be given by Proposition \ref{topar}.
Setting
\begin{equation*}
\zeta=\min \left\{|p-K^{\pm 1}(p)|: p\in [\eta,1/2-\eta]\cup [1/2+\eta, 1-\eta]\right\}.
\end{equation*}
Since \( K \) is a Morse–Smale diffeomorphism on \( \mathbb{S}^1 \) with exactly two fixed points, it follows that \( \zeta > 0 \).
Choose an open neighborhood $\U_2\subset \U_1$ of $f_0$ such that
$$
d(f^{\pm 1}(x),f_0^{\pm 1}(x))<\min\left\{\frac{\zeta}{2}, \frac{(1-\theta)}{2}\eta\right\}
$$
for every $x\in \TT^5$ and $f\in \U_2$.
Here, the metric \( d \) can also be understood as the standard metric on the lifted space \( \mathbb{R}^5 \).

\medskip
Given $f\in \U_2$, let us show the contraction on the set $\TT^4\times (-\eta,\eta)$ firstly.
Let $x=(a,b)\in \TT^4\times [-\eta,\eta]$. As $f_0(\TT^4\times \{0\})=\TT^4\times \{0\}$, $f_0(a,0)\in \TT^4\times \{0\}$, it follows that
\begin{eqnarray*}
d(f(a,b),\TT^4\times \{0\})&\le& d(f(a,b), f(a,0))+d(f(a,0),f_0(a,0))\\
&+ & d(f_0(a,0), \TT^4\times \{0\})\\
&\le & d(f(a,b), f(a,0))+\frac{(1-\theta)}{2}\eta.
\end{eqnarray*}
We assume without loss of generality that $(a,b)$ is contained in the line segment
$$
L_a=\{a\}\times [0,\eta]\subset \TT^4\times [-\eta,\eta],
$$
which is tangent to $T\SS^1$ everywhere. Take \eqref{exttonei} into account, we get
\begin{eqnarray*}
d(f(a,b), f(a,0))&\le& {\rm length}(f(L_a))\\
&\le &\sup_{x\in L_a}||Df|_{T_x\SS^1}\|\cdot {\rm length}(L_a)\\
&< & \theta\eta.
\end{eqnarray*}
Consequently, we have 
$$
|\pi_{\SS^1}(f(a,b))|\le d(f(a,b),\TT^4\times \{0\})<\theta\eta+\frac{1-\theta}{2}\cdot\eta=\frac{1+\theta}{2}\cdot\eta,
$$
thus,
$$
f\left(\TT^4\times [-\eta,\eta]\right)\subset \TT^4\times \left(-\frac{1+\theta}{2}\eta,\frac{1+\theta}{2}\eta\right)\subset \TT^4\times (-\eta,\eta) .
$$
This implies that 
$$
\mathcal{A}_f=\bigcap_{n=1}^{\infty}f^n\left(\TT^4\times (-\eta,\eta)\right)
$$
is an attractor of $f$. So, we have 
\begin{equation}\label{gopositive}
f^n(x)\to \mathcal{A}_f, \quad \forall~x\in \TT^4\times (-\eta,\eta).
\end{equation}
Using the same argument, we can show that 
$$
\mathcal{R}_f=\bigcap_{n=1}^{\infty}f^{-n}\left(\TT^4\times (1/2-\eta,1/2+\eta)\right)
$$
is a repeller of $f$, i.e., attractor of $f^{-1}$ with the property
\begin{equation}\label{gonegative}
f^{-n}(x)\to \mathcal{R}_f, \quad \forall~x\in \TT^4\times (1/2-\eta,1/2+\eta).
\end{equation}
By the contractions on determines 
$$
\max_{x\in \TT^4\times [-\eta,\eta]}|{\rm det} Df(x)|<\theta,\quad \max_{x\in \TT^4\times [1/2-\eta,1/2+\eta]}|{\rm det} Df^{-1}(x)|<\theta,
$$
we know that 
$$
{\rm Leb}(\mathcal{A}_f)={\rm Leb}(\mathcal{R}_f)=0.
$$

To see \eqref{orbit} we need to analysis the dynamical behavior of points in $\TT^4\times [\eta,1/2-\eta]\cup [1/2+\eta, 1-\eta]$.

\begin{claim}\label{eno}
For any $f\in \U_2$ and $x\in \TT^4\times [\eta,1/2-\eta]\cup [1/2+\eta, 1-\eta]$, there exist $n_1,n_2\in \NN$ such that 
$$
f^{n_1}(x)\in  \TT^4\times (-\eta,\eta), \quad f^{-n_2}(x)\in  \TT^4\times (1/2-\eta,1/2+\eta).
$$
\end{claim}

\begin{proof}
Suppose $x\in \TT^4\times [\eta,1/2-\eta]$ without loss of generality. 
Recall that $\pi_{\SS^1}$ is the projection map from $\TT^4\times \SS^1$ to $\SS^1$. 
One can check directly that 
$$
K(\pi_{\SS^1}(x))=\pi_{\SS^1}(f_0(x)),\quad \forall~x\in \TT^5.
$$ 
Observe also that
$$
\pi_{\SS^1}(x)-K(\pi_{\SS^1}(x))\ge \zeta,\quad d(f(x),f_0(x))\le \zeta/2,
$$ 
Altogether, we get
\begin{eqnarray*}
\pi_{\SS^1}(x)-\pi_{\SS^1}(f(x)) &=& \left(\pi_{\SS^1}(x)-K(\pi_{\SS^1}(x))\right)-(\pi_{\SS^1}(f(x))-\pi_{\SS^1}(f_0(x)))\\
&\ge & \left(\pi_{\SS^1}(x)-K(\pi_{\SS^1}(x))\right)-|\pi_{\SS^1}(f(x))-\pi_{\SS^1}(f_0(x))|\\
&\ge & \zeta-d(f(x),f_0(x))\\
&\ge & \zeta-\zeta/2\\
&=& \zeta/2.
\end{eqnarray*}
Therefore, we have
$
\pi_{\SS^1}(f^{n+1}(x))\le \pi_{\SS^1}(f^n(x))-\zeta/2
$
as long as $f^n(x)\in \TT^4\times [\eta,1/2-\eta]$. This implies that there exists some $n_1\in \NN$ so that 
$$
f^{n_1}(x)\in \TT^4\times [0,\eta)\subset \TT^4\times (-\eta,\eta).
$$ 
By considering $f^{-1}, f_0^{-1}, K^{-1}$ instead, with the similar argument we can show that there exists $n_2\in \NN$ so that
$$
f^{-n_2}(x)\in  \TT^4\times (1/2-\eta,1/2+\eta).
$$
\end{proof}

We continue the proof of Proposition \ref{invariant}.
Fix any point $x\notin \mathcal{A}_f\cup \mathcal{R}_f$, we show \eqref{orbit} by its location into three cases:
\begin{itemize}
\item When $x\in\TT^4\times [\eta,1/2-\eta]\cup [1/2+\eta, 1-\eta]$, by Claim \ref{eno}, there are $n_1,n_2\in \NN$ such that 
$$
f^{n_1}(x)\in  \TT^4\times (-\eta,\eta), \quad f^{-n_2}(x)\in  \TT^4\times (1/2-\eta,1/2+\eta).
$$
Using \eqref{gopositive} and \eqref{gonegative} for points $f^{n_1}(x)$ and $f^{-n_2}(x)$ respectively, we see
$$
f^n(x)\to \mathcal{A}_f,\quad f^{-n}(x)\to \mathcal{R}_f
$$
as $n$ goes to infinity. 
\smallskip
\item
When $x\in \TT^4\times (-\eta,\eta)$, we know from \eqref{gopositive} that $f^n(x)\to \mathcal{A}_f$ as $n\to \infty$. Since $x\notin \mathcal{A}_f$, there is $m\in \NN$ satisfying
$
f^{-m}(x)\notin \TT^4\times (-\eta,\eta).
$
If $f^{-m}(x)\in \TT^4\times (1/2-\eta,1/2+\eta)$, then we have
$f^{-n}(x)\to \mathcal{R}_f$ as $n\to +\infty$ by \eqref{gonegative};
otherwise, 
$$f^{-m}(x)\in \TT^4\times [\eta,1/2-\eta]\cup [1/2+\eta, 1-\eta],$$ 
then some negative iterate of $f^{-m}(x)$ will go to $\TT^4\times (1/2-\eta,1/2+\eta)$, which case we have checked.
\smallskip
\item
When $x\in \TT^4\times (1/2-\eta,1/2+\eta)$, we get the desired result by observing $x\in \mathcal{R}_f$, using the similar argument as in the above case.
\end{itemize}

Now we complete the proof of Proposition \ref{invariant}.

\end{proof}

\begin{remark}
Using the invariance of $\TT^4\times \{a_1\}$ and $\TT^4\times \{a_2\}$ under $f_0$, one can show with the argument in the proof of Proposition \ref{invariant} that 
$$
\mathcal{A}_{f_0}=\TT^4\times \{a_1\},\quad  \mathcal{R}_{f_0}=\TT^4\times \{a_2\}.
$$
Moreover, we can verify that for any $\varepsilon>0$, there exists a $C^0$ neighborhood $\mathcal{V}$ of $f_0$ such that 
$$
\mathcal{A}_f\subset B(\mathcal{A}_{f_0},\varepsilon),\quad \mathcal{R}_f\subset B(\mathcal{R}_{f_0},\varepsilon).
$$
whenever $f\in \V$. Here \( B(\mathcal{A}_{f_0}, \varepsilon) \) denotes the \( \varepsilon \)-neighborhood of \( \mathcal{A}_{f_0} \), and \( B(\mathcal{R}_{f_0}, \varepsilon) \) is similarly the \( \varepsilon \)-neighborhood of \( \mathcal{R}_{f_0} \).
\end{remark}

As a consequence of Proposition \ref{invariant}, we have the following two results:

\begin{corollary}\label{invm}
Let $f\in \U_2$. Any $f$-invariant measure $\mu$ is supported on $\mathcal{A}_f\cup \mathcal{R}_f$, which satisfies 
\begin{equation}\label{2g}
|\chi(x,E_{2}^c)|\ge -\log \theta,\quad \mu-\textrm{a.e.}~x.
\end{equation}
\end{corollary}

\begin{proof}
This can be deduced from \eqref{orbit} and \eqref{nu2} of Proposition \ref{invariant}.
\end{proof}

\begin{corollary}\label{gue2}
For any $f\in \U_2$, there exists a subset $\Gamma_f$ of full Lebesgue measure in $\TT^5$ such that for any $x\in \Gamma_f$, any $\mu\in \mathscr{M}(x,f)$ is a Gibbs $u$-state supported on $\mathcal{A}_f$, satisfying
\begin{equation}\label{pwl}
\chi(x,E_{2}^c)\le \log \theta,\quad \mu-\textrm{a.e.}~x.
\end{equation}
\end{corollary}

\begin{proof}
By \eqref{ufull} in Proposition \ref{basicu}, for any $f\in \U_2$, there exists a full Lebesgue measure subset $G_f$ of $\TT^5$ such that $\mathscr{M}(x,f)\subset \G^u(f)$ for any $x\in G_f$. Let us take
$$
\Gamma_f=G_f\setminus (\mathcal{A}_f\cup \mathcal{R}_f).
$$
By \eqref{zerol} of Proposition \ref{invariant}, we have ${\rm Leb}(\mathcal{A}_f\cup \mathcal{R}_f)=0$. Thus, $\Gamma_f$ has full Lebesgue measure in $\TT^5$ as well. As we have shown in \eqref{orbit},
$$
f^n(x)\to \mathcal{A}_f ~\textrm{as}~ n\to +\infty, ~\forall~ x\in \Gamma_f.
$$
This guarantees that for any $x\in \Gamma_f$, any $\mu\in \mathscr{M}(x,f)$ is a Gibbs $u$-state supported on $\mathcal{A}_f$. One gets \eqref{pwl} by \eqref{nu2} of Proposition \ref{invariant}.
%
\end{proof}

\subsubsection{Mostly expanding along $E_1^c$}

The next result asserts that any $C^1$ diffeomorphism $C^1$-close to $f_0$ are mostly expanding along the stronger sub-bundle $E_1^c$. i.e., every Gibbs $u$-state of $f$ admits uniformly positive Lyapunov exponents along $E_1^c$. 

\begin{proposition}\label{ge}
There exists a real number $\beta>0$ and a $C^1$ open neighborhood $\U_3$ of $f_0$ such that
$$
\int\log \|Df|_{E_{1}^c}\|{\rm d}\mu>\beta,\quad \forall~ \mu\in \mathcal{G}^u(f), \quad \forall~ f\in \U_3.
$$
\end{proposition}

To show Proposition \ref{ge}, we need the next observation.

\begin{lemma}\label{u0g}
We have that
$$
\mu(U_0)<\frac{\log \lambda}{\log 3 \lambda},\quad \forall~ \mu\in \mathcal{G}^u(f_0).
$$
\end{lemma}

\begin{proof}
By ergodic decomposition theorem (see e.g. \cite[Theorem 5.13]{VO16}) and \eqref{cpu} in Proposition \ref{basicu}, it suffices to verify this result for ergodic Gibbs $u$-states of $f_0$. To this end, let us introduce the projection $\pi: \TT^2\times \TT^2\times \SS^1\to \TT^2: (a,b,c)\mapsto a$.
We see from the definition of $f_0$ that
\begin{equation}\label{comm}
\pi \circ f_0=A^{2}\circ \pi.
\end{equation}
%
For any given $\mu\in \G_{erg}^u(f_0)$, we claim that $\pi_{\ast}\mu\in \G_{erg}^u(A^{2})$. The ergodicity of $\pi_{\ast}\mu$ can be checked by definition directly, now we show $\pi_{\ast}\mu$ is a Gibbs $u$-state of $A^{2}$.
Since the ergodicity of $\mu$ implies that the basin $B(\mu,f)$ exhibits full $\mu$-measure, one can find a subset $D'$ contained in a local strong unstable manifold with full Lebesgue measure so that 
$$
\lim_{n\to\infty}\frac{1}{n}\sum_{i=0}^{n-1}\delta_{f_0^i(x)}=\mu, \quad \forall~ x\in D'.
$$
Since $\pi$ sends each strong unstable manifold of point in $\{a\}\times \TT^2\times \SS^1$ w.r.t.$f_0$ to the unique strong unstable manifold of $a$ w.r.t. $A^{2}$. Restricted on each strong unstable manifold w.r.t. $f_0$, the projection $\pi$ is a smooth diffeomorphism.  Using \eqref{comm}, the above convergence implies
$$
\lim_{n\to\infty}\frac{1}{n}\sum_{i=0}^{n-1}\delta_{A^{2i}(\pi(x))}=\pi_{\ast}\mu, \quad \forall~ x\in D'.
$$
Observe that $\pi(D')$ admits positive Lebesgue measure inside the corresponding strong unstable manifold of $A^{2}$. By applying \cite[Theorem C]{CYZ18}(or \cite[Theorem A]{HYY19}), we know $\pi_{\ast}\mu$ is a Gibbs $u$-state, thus $\pi_{\ast}\mu\in \G_{erg}^u(A^{2})$.

Observe that ${\rm Leb}_{\TT^2}$ is the unique Gibbs $u$-state of $A^{2}$, we know
$$
\pi_{\ast}\mu={\rm Leb}_{\TT^2}.
$$
Recalling that $U_0= [-\delta_0,\delta_0]^5$, and applying (H\ref{three}), this yields
\begin{eqnarray*}\label{}
\mu(U_0)&\le &\mu([-\delta_0,\delta_0]^2\times \TT^2\times \SS^1)\\
&=&\pi_{\ast}\mu([-\delta_0,\delta_0]^2)\\
&=&{\rm Leb}_{\TT^2}([-\delta_0,\delta_0]^2)\\
&< &\frac{\log \lambda}{\log 3 \lambda}.
\end{eqnarray*}
This completes the proof.
\end{proof}

Now we can give the proof of Proposition \ref{ge} as follows:

\begin{proof}[Proof of Proposition \ref{ge}]
To show the result, it suffices to find a constant $\beta>0$ such that 
\begin{equation}\label{11g}
\int\log \|Df_0|_{E_{1}^c}\|{\rm d}\mu>\beta,\quad \forall ~\mu\in \mathcal{G}_{erg}^u(f_0).
\end{equation}
Indeed, by applying \eqref{cpu} in Proposition \ref{basicu}, we conclude from \eqref{11g} that 
\begin{equation}\label{22g}
\int\log \|Df_0|_{E_{1}^c}\|{\rm d}\mu>\beta,\quad \forall ~\mu\in \mathcal{G}^u(f_0).
\end{equation}
This gives the desired result for $\beta$.
If the conclusion is not true, then there exists a sequence of diffeomorphisms $f_n$ that converges to $f_0$ in $C^1$-topology, 
and $\mu_n\in \G^u(f_n)$ satisfying
$$
\int\log \|Df_n|_{E_{1}^c}\|{\rm d}\mu_n\le \beta.
$$
Without loss of generality, we assume $\mu_n\to \mu$ as $n\to +\infty$. Since $\log \|Df_n|_{E_{1}^c}\|$ converges to $\log \|Df_0|_{E_{1}^c}\|$ uniformly, which together with the convergence $\mu_n\to \mu$ gives
$$
\int\log \|Df_0|_{E_{1}^c}\|{\rm d}\mu=\lim_{n\to +\infty}\int\log \|Df_n|_{E_{1}^c}\|{\rm d}\mu_n\le \beta.
$$
Thus, we get a contraction to \eqref{22g}.

It remains to verify \eqref{11g}.
By \eqref{peb} in Lemma \ref{splitting}, we get
$$
\|Df_0|_{E_1^c(x)}\|=\frac{\partial }{\partial x_2}P_k(x)\ge 1/2, \quad x\in U_0.
$$
On the other hand, we have $\|Df_0|_{E_1^c}\|=\lambda$ whenever $x\notin U_0$. Thus, for any $\mu\in \G^u(f_0)$ we have
\begin{eqnarray*}
\int \log \|Df_0|_{E_1^c}\|{\rm d}\mu&=&\int_{U_0}\log \|Df_0|_{E_1^c}\|{\rm d}\mu+\int_{\TT^5\setminus {U_0}}\log \|Df_0|_{E_1^c}\|{\rm d}\mu\\
&\ge& \mu(U_0)\log \frac{1}{2}+(1-\mu(U_0))\log \lambda\\
&=&\mu(U_0)\left(\log \frac{1}{2}-\log \lambda\right)+\log \lambda\\
&=&\mu(U_0)\cdot(-\log(2\lambda))+\log\lambda.
\end{eqnarray*}
By Lemma \ref{u0g}, we get 
$$
\int \log \|Df_0|_{E_1^c}\|{\rm d}\mu>\left(1-\frac{\log 2\lambda}{\log 3\lambda}\right)\cdot \log \lambda.
$$
Thus, we conclude the result after taking 
$$
\beta=\left(1-\frac{\log 2\lambda}{\log 3\lambda}\right)\cdot \log \lambda>0.
$$
\end{proof}


\subsubsection{Proof of Theorem \ref{gtop}}

Take $\U=\U_2\cap\U_3$, where $\U_2$ and $\U_3$ are the $C^1$ open neighborhoods of $f_0$ provided in Proposition \ref{invariant} and Proposition \ref{ge}, respectively. Therefore, any $f\in \U$ exhibits the partially hyperbolic splitting
$$
T(\mathbb{T}^5)=E^{u}\oplus_{\succ} E_{1}^c\oplus_{\succ} E_{2}^c\oplus_{\succ} E^{s}.
$$
We check firstly that $f\in \U$ satisfies the properties \eqref{nonun} to \eqref{empirical} as follows:
\begin{enumerate}[(1)]
\smallskip
\item 
According to \eqref{E1nh} of Lemma \ref{first} and property \eqref{nu2} of Proposition \ref{invariant}, one knows that, up to shrinking $\U$, the central sub-bundle $E_{1}^c, E_{2}^c$ are no uniformly contracting/expanding.  
\smallskip
\item
Item \eqref{dimension} is guaranteed by the construction of $f_0=F_k$ in Lemma \ref{first} and the $C^1$-robustness of partially hyperbolic splitting. 
\smallskip
\item
Now we show Item \eqref{hyperbolic}. For any $f\in \U$, let $\mu\in \G^u(f)$. From Proposition \ref{ge}, it follows that all the Lyapunov exponents along $E_1^c$ of $\mu$ are larger than some constant $\beta>0$. By Corollary \ref{invm}, the Lyapunov exponents of $\mu$ along $E_2^c$ are bounded away from zero. 
\smallskip
\item\label{fpl} By Corollary \ref{gue2},  for ${\rm Leb}$-a.e. $x\in \mathbb{T}^5$, any $\mu\in \mathscr{M}(x,f)$ is a Gibbs $u$-state supported on $\mathcal{A}_f$, which also admits Lyapunov exponents along $E_2^c$ smaller than $\log \theta<0$. By Proposition \ref{ge}, the Lyapunov exponents of $\mu$ along $E_1^c$ are larger than $\beta>0$.  
\end{enumerate}

It remains to give the proof of \eqref{uniqueness}. 
By \cite[Theorem B]{MC25}, any $C^{1+}$ diffeomorphism $f\in\U$ exhibits finitely many physical measures, which are  ergodic SRB measures whose union of basins has full Lebesgue measure in $\TT^5$. 
Let $\mu$ be a physical measure of $C^{1+}$ diffeomorphism $f\in \U$. Since $B(\mu,f)$ admits positive Lebesgue measure and $\mathscr{M}(x,f)=\{\mu\}$ for every $x\in B(\mu,f)$, we know from \eqref{fpl} that $\mu$ is supported on the attractor $\mathcal{A}_f$, and it admits only positive Lyapunov exponents along $E_{1}^c$ and negative Lyapunov exponents along $E_{2}^c$. 

Thus, it suffices to show that, up to reducing $\U$, any $C^{1+}$ diffeomorphism admits only one physical measure.
We need the next observation.

\begin{claim}\label{sf}
There exists \( \rho > 0 \) such that for every \( C^{1+} \) diffeomorphism \( f \in \mathcal{U} \) and every physical measure \( \mu \) of \( f \), there exists an unstable disk \( D_{\rho}(\mu)\subset \mathcal{A}_f \) of size \( \rho \) such that Lebesgue almost every point in \( D_{\rho}(\mu) \) belongs to the basin \( B(\mu, f) \).
\end{claim}

As $f_0$ preserves the set $\TT^4\times \{a_1\}$, let us consider the diffeomorphism $\widehat{f}_0=f_0|_{\TT^4\times \{a_1\}}$.
By construction, $\widehat{f}_0$ is a skew product over $\TT^2\times \TT^2$ of the form $\widehat{f}_0(x,y)=(A^{2}(x),h_x(y))$ satisfying
\begin{itemize}
\item $A^{2}$ is a transitive Anosov diffeomorphism,
\item $h_p=A$ for the fixed point $p\in \TT^2\setminus [-\delta_0,\delta_0]^2$ of $A$ (recall (H\ref{pou})).
\end{itemize}
One can check that the stable manifold of the fixed point $(p,p)\in \TT^4$ of $\widehat{f}_0$ is dense in $\TT^4$. Denote by ${\bf{p}}=(p,p,a_1)\in \TT^2\times \TT^2\times \SS^1=\TT^5$, which is a hyperbolic fixed point of $f_0$. As $f_0$ is uniformly contracting along $E_2^c$ inside $\TT^4\times [-\eta,\eta]$, one thus concluds that the global stable manifold $W^s({\bf{p}},f)$ is dense in $\TT^4\times [-\eta,\eta]$. Therefore, for any $x\in \TT^4\times [-\eta,\eta]$, if $\mathcal{D}_x$ is a disk of radius $\rho$ centered at $x$, and transverse to $E_2^c\oplus E^s$, then there is $L_x>0$ such that $\mathcal{D}_x$ intersects $W^s_{L_x}({\bf{p}},f_0)$ transversely at some point in the interior of $W^s_{L_x}({\bf{p}},f_0)$. By compactness, we can find a uniform $L>0$ such that 
$$
\mathcal{D}\pitchfork W^s_{L}({\bf{p}},f_0)\neq\emptyset
$$
for every disk $\mathcal{D}$ of size $\rho$ and transverse to $E^c_2\oplus E^s$.
By applying the continuity of stable manifold of fixed point w.r.t. diffeomorphisms, one knows that up to reducing $\U$ if necessary, for the continuation ${\bf{p}}_f$ of every $f\in \U$, it holds that
$$
\mathcal{D}\pitchfork W^s_{L}({\bf{p}}_f,f)\neq\emptyset
$$
for every disk $\mathcal{D}$ of size $\rho$ and transverse to $E_2^c\oplus E^s$. 

Given any $C^{1+}$ diffeomorphism $f\in \U$, to show the uniqueness of physical measures, we assume by contradiction that there exist two physical measures $\mu_1$ and $\mu_2$ of $f$. 
By Claim \ref{sf}, there are two unstable disks $D_{\rho}(\mu_1)$ and $D_{\rho}(\mu_2)$ associated to $\mu_1$ and $\mu_2$ with the following properties:
\begin{itemize}
\item $D_{\rho}(\mu_i)$ is contained in the attractor $\mathcal{A}_f$ for each $i=1,2$;
\smallskip
\item Lebesgue almost every point of $D_{\rho}(\mu_i)$ is contained in $B(\mu_i,f)$ for each $i=1,2$. 
\end{itemize}
On the other hand, as both  $D_{\rho}(\mu_1)$ and $D_{\rho}(\mu_2)$ are tangent to $E^u\oplus E_1^c$, so they are transverse to $E_2^c\oplus E^s$, it follows that 
$$
D_{\rho}(\mu_i)\pitchfork W^s_{L}({\bf{p}}_f,f)\neq\emptyset,\quad i=1,2.
$$
By applying the Inclination lemma \cite[Theorem 5.7.2]{B02}, there exist $n_0\in \NN$ and disks $D_i\subset f^{n_0}(D_{\rho}(\mu_i))$, $i=1,2$ that are $C^1$ close to $W^s_{L}({\bf{p}}_f,f)$ with size much smaller than the uniform size of local stable manifolds in $\mathcal{A}_f$. Then, by the absolute continuity of the stable foliation, each $D_i$ intersects the local stable manifolds from $\mathcal{A}_f$ transversely with positive Lebesgue measure. Observe also that the stable manifold can only lie in the same basin, and Lebesgue almost every point of $D_i$ is contained in $B(\mu_i,f)$ for each $i=1,2$, we conclude that ${\rm Leb}_{D_1}(B(\mu_2,f))>0$ and vice versa, which leads a contradiction.

%
%
%

\medskip
We then left to show the claim.

\begin{proof}[Proof of Claim \ref{sf}]
Let $\beta>0$ be the constant given by Proposition \ref{ge}. For every $f\in \U$, define
$$
\Lambda_{\beta}(f)=\{x\in \TT^5: \prod_{i=0}^{n-1}\|Df^{-1}|_{E_1^c(f^{-i}(x))}\|\le {\rm e}^{-\beta n},\quad \forall~ n\in \NN\}.
$$
By applying the result \cite[Lemma 3.2]{MC21}, one knows that up to reducing $\U$, there exists a uniform $r_{\beta}>0$ such that for any $f\in \U$, any $x\in \Lambda_{\beta}(f)$ admits a local unstable manifold $W^u_{r_{\beta}}(x)$ tangent to $E^u\oplus E_{1}^c$ of size $r_{\beta}$.

Let us fix $f\in {\rm Diff}^{1+}(\TT^5)\cap \U$. For any physical measure $\mu$ of $f$, since it is also a Gibbs $u$-state, following Proposition \ref{ge}, it holds that
$
\chi(E_1^c,\mu)>\beta,
$
which implies 
\begin{equation}\label{touni}
\mu(\Lambda_{\beta}(f))>0,
\end{equation}
by applying the result \cite[Lemma 4.5]{MWC24}. We then have $\mu(B(\mu,f)\cap \Lambda_{\beta}(f))>0$ as $\mu$ is ergodic. 
Observe that \( \mu \) is also an ergodic SRB measure with unstable index $2$, its conditional measures along Pesin unstable manifolds tangent to $E^u\oplus E_1^c$ are, in fact, equivalent to the Lebesgue measures on these manifolds. Therefore, one can find an unstable disk \( D_{\rho}(\mu) \subset W^u_{r_{\beta}}(x) \) of radius \( \rho \), centered at some point in \( \Lambda_{\beta}(f) \), such that Lebesgue almost every point in $D_{\rho}(\mu)$ is contained in the basin of \( \mu \), where \( \rho \) is chosen independent of \( f \) and \( \mu \) (see \cite[Lemma 4.6]{MC21} for more details).
\end{proof}

\section{Diffeomorphisms within two  physical measures of different unstable indices}\label{twop}

The main goal of this section is to prove Theorem \ref{maintwo} for the case $m=2$, which is deduced from the following result within more details. 

\begin{theorem}\label{yizhi}
There exists $f_0\in {\rm Diff}^{\infty}(\TT^5)$ and a \( C^1 \) open neighborhood $\U$ of $f_0$ such that any $f\in\U$ exhibits the following properties:
\begin{enumerate}[(1)]
\item\label{psr} $f$ admits a partially hyperbolic splitting $$T(\TT^5)=E^{u}\oplus_{\succ} E_1^c\oplus_{\succ} E_2^c\oplus_{\succ} E^{s}
$$
with ${\rm dim}E^{u}={\rm dim}E_1^c={\rm dim}E_2^c=1$;
\item\label{noh} neither $E_1^c$ nor $E_2^c$ is uniformly contracting/expanding;
\smallskip
\item\label{hyu} every invariant measure of $f$ is hyperbolic;
\smallskip
\item\label{twofen} if $f$ is $C^{1+}$, then there exists a full Lebesgue measure subset \( B_1 \cup B_2 \) of $\TT^5$ with \( {\rm Leb}(B_1) > 0 \) and \( {\rm Leb}(B_2) > 0 \) such that:  
\begin{itemize}
\item[--] for any $x\in B_1$, for any $\mu\in \mathscr{M}(x,f)$, it holds that 
$$
\chi(z, E_1^c)<0,\quad \mu-\textrm{a.e.}~z,
$$
\item[--] for any $x\in B_2$, for any $\mu\in \mathscr{M}(x,f)$, it holds that 
$$
\chi(z, E_1^c)>0>\chi(z, E_{2}^c),\quad \mu-\textrm{a.e.}~z;
$$
\end{itemize}
\item \label{ptw} if $f$ is $C^{1+}$, then it has two ergodic hyperbolic physical measures $\mu_1,\mu_2$ of unstable indices $1$ and $2$ respectively, and 
$$
{\rm Leb}\left(B(\mu_1,f)\cup B(\mu_2,f)\right)={\rm Leb}(\TT^5).
$$
\end{enumerate}
\end{theorem}

\subsection{Construction of $\{G_k\}$}\label{Gk}

Let $J$ be a $C^{\infty}$ Morse-Smale diffeomorphism on $\mathbb{S}^1=[0,1]/\{0,1\}$ whose non-wandering set is formed by 
$$
b_1=0, \quad b_2=1/4,\quad b_3=1/2,\quad b_4=3/4,
$$
where $b_1,b_3$ are sinks and $b_2,b_4$ are sources. Assuming also that
\begin{enumerate}[(J1)]
\item \label{j1}
$$
\lambda^{-1}<\|DJ(x)\|<\lambda,\quad x\in \SS^1,
$$
\item\label{j2}
$$
\|DJ(x)\|<\frac{3}{2}\lambda^{-1},\quad x\in [b_i-\delta_0,b_i+\delta_0],\quad i=1,3;
$$
\item\label{j3}
$$
\|DJ(x)\|>\frac{2}{3}\lambda,\quad x\in [b_i-\delta_0,b_i+\delta_0],\quad i=2,4.
$$
\end{enumerate}
Here $\lambda=\lambda_0^N$ is the constant satisfying (H\ref{one})-(H\ref{three}).

Define $G=B\times J$ on $\TT^5$, which admits the partially hyperbolic splitting 
$$
T(\TT^5)={\mathcal{E}}^{uu}\oplus {\cE}^{u}\oplus T\mathbb{S}^1 \oplus {\cE}^{s} \oplus {\cE}^{ss}.
$$
Let us take  
$$
U_0= [-\delta_0, \delta_0]^5=\F^{uu}_{\delta_0}(0)\times \F^{u}_{\delta_0}(0)\times \SS^1_{\delta_0}(0) \times \F^s_{\delta_0}(0)\times \F_{\delta_0}^{ss}(0),
$$
where we will modify $G$ to generate a sequence of $C^{\infty}$ diffeomorphisms $\{G_k\}$. For any fixed $k\in \NN$, we construct $G_k$ in the following way:
\begin{itemize}
\smallskip
\item $G_k$ agrees with $G$ for points outside $U_0$, 
\smallskip
\item if $x=(x_1,x_2,x_3,x_4,x_5)\in U_0$,
put
$$
G_k(x)=\left(\lambda^2 x_1,Q_k(x),J(x_3),\lambda^{-1}x_4,\lambda^{-2}x_5\right),
$$
where 
$$
Q_k(x) = \varphi(kx_2)\cdot \varphi(x_3)\cdot \varphi(x_4) \cdot\left( \frac{1}{2}- \lambda\right)x_2 + \lambda x_2.
$$
\end{itemize}

By the construction of $Q_k$, the choice of $\lambda$ with the fact $R(x,y,z)\ge -C$ (recall (H\ref{two})), arguing similarly to the proof of Lemma~\ref{splitting},  we conclude the following result.

\begin{lemma}\label{splittingtwo}
We have the following result for every $Q_k$:
$$
\frac{\partial}{\partial x_2}Q_k(x)\in \left[\frac{1}{2},\frac{\lambda^2}{3}\right], \quad \forall~x\in U_0,\quad
\frac{\partial}{\partial x_2}Q_k(0)=\frac{1}{2}.
$$
$$
\frac{\partial }{\partial x_i}Q_k(x)=0~ \textrm{if}~x_i=0, \quad
 \lim_{k\to \infty}\sup_{x\in U_0}\left\{\frac{\partial }{\partial x_i}Q_k(x)\right\}=0,\quad i=3,4.
$$
\end{lemma}

One can deduce the following result by construction.

\begin{lemma}\label{second}
For any $\varepsilon>0$, there exists $\widehat{k}_{\varepsilon}$ such that for any $k\ge \widehat{k}_{\varepsilon}$, $G_k$ is a $C^{\infty}$ diffeomorphism admitting the partially hyperbolic splitting 
$$
T(\mathbb{T}^5)=E^{u}\oplus_{\succ} E_1^c\oplus_{\succ} E_2^c\oplus_{\succ} E^{s}
$$
such that
\begin{enumerate}[(1)]
\item $E^u={\cE}^{uu}$ and is uniformly expanding;
\smallskip
\item $E^s\subset \mathscr{C}_{\varepsilon}({\cE}^{s}\oplus {\cE}^{ss}, {\cE}^{u}\oplus T\SS^1)$ and is uniformly contracting;
\item \label{E1nh}$E_1^c= {\cE}^u$, which satisfies 
$\|DF_k|_{E_1^c(x)}\|=\lambda>1$ if $x\notin U_0$ and equals to $1/2$ at $0\in \TT^5$;
\item $E_2^c\subset \mathscr{C}_{\varepsilon}(T\SS^1, {\cE}^u)$.
\end{enumerate}
\end{lemma}

\begin{proof}
We see from Lemma \ref{splittingtwo} that
$$
\frac{\partial}{\partial x_2}Q_k(x)\ge \frac{1}{2}>0,
$$
with the same argument as in the proof of Lemma \ref{Cinfinite}, one can prove that every $G_k$ is a $C^{\infty}$ diffeomorphism.

We have the expression of $DG_k(x)$ on $U_0$ as follows:
\begin{equation*}\label{exp}
\renewcommand{\arraystretch}{1.3}
\begin{pmatrix}
\lambda^2 & 0 & 0 & 0 & 0 \\
0 & \frac{\partial }{\partial x_2}Q_k(x) & \frac{\partial }{\partial x_3}Q_k(x) & \frac{\partial }{\partial x_4}Q_k(x) & 0 \\
0 & 0 & DJ(x_3) & 0 & 0 \\
0 & 0 & 0 & \lambda^{-1} & 0 \\
0 & 0 & 0 & 0 & \lambda^{-2}
\end{pmatrix}
\end{equation*}
Note also that $DG_k$ coincides with $DG$ outside $U_0$.
This implies the $DG_k$-invariance of ${\cE}^{uu}$, ${\cE}^{u}\oplus T\SS^1\oplus {\cE}^{s}\oplus {\cE}^{ss}$, and $\cE^u\oplus T\SS^1$. With the similar argument as in the proof of Lemma \ref{first},  one knows that for every $\varepsilon>0$, there exists $\widehat{k}_{\varepsilon}\in \NN$ so that for any $k\ge \widehat{k}_{\varepsilon}$,  
one can find $DF_k$-invariant sub-bundles $E^u={\cE}^{uu}$,  $E^s\subset \mathscr{C}_{\varepsilon}({\cE}^{s}\oplus {\cE}^{ss}, {\cE}^{u}\oplus T\SS^1)$, $E_1^c= {\cE}^u$ and $E_2^c\subset \mathscr{C}_{\varepsilon}(T\SS^1, {\cE}^u)$ to satisfy the desired property. 
\end{proof}

\subsection{DA diffeomorphisms}

Recall that $A: \TT^2\to \TT^2$ is a linear Anosov diffeomorphism with eigenvalues $\lambda^{-1}<1<\lambda$, which admits a uniform hyperbolic splitting 
$T(\TT^2)=\cE^{u}\oplus \cE^s$. Denote by 
$$
V_0=\F_{\delta_0}^{u}(0)\times \F_{\delta_0}^s(0),
$$
where $\F_{\delta_0}^{u}(0)$ and $\F_{\delta_0}^{u}(0)$ are local unstable and stable manifolds of size $\delta_0$ at $0$, respectively. 

The following DA diffeomorphism is constructed by deforming $A$ on a neighborhood of the saddle point $0$, making it to be a sink. The proof follows similarly from the classical approach (see \cite[Section 17.2]{KH} or \cite[Section 7.8]{Robin}), we provide a proof in Appendix \ref{app} for completeness.

\begin{lemma}\label{DAz}
For each $k\in \NN$, let $g_k:\TT^2\to \TT^2$ be a diffeomorphism defined by
$$
g_k|_{\TT^2\setminus V_0}=A|_{\TT^2\setminus V_0},
$$
$$
g_k(u,v)=A(u,v)+\left(\varphi(ku)\cdot \varphi(v)\cdot \left(\frac{1}{2}-\lambda\right)u,0\right), ~\forall~(u,v)\in V_0.
$$
Then for any $\varepsilon>0$, there exists $\kappa_{\varepsilon}$ such that for any $k\ge \kappa_{\varepsilon}$, $g_k$ admits a partially hyperbolic splitting $T(\TT^2)=E^{cu}\oplus E^s$ for which, $E^s\subset \mathcal{C}_{\varepsilon}(\cE^s,\cE^u)$ is uniformly contracting and $E^{cu}=\cE^u$. Moreover, there exists an open neighborhood $V\subset V_0$ of $(0,0)\in \TT^2$ with the following properties: 
\begin{enumerate}[(1)]
\item \label{at} $g_k(\overline{V})\subset V$;
\smallskip
\item \label{convg} $\lim_{n\to +\infty}g_k^n(u,v)=(0,0)$ whenever $(u,v)\in V$;
\smallskip
\item \label{hyper} $\TT^2\setminus\bigcup_{n\in \NN}g_k^{-n}(V)$ is a compact invariant hyperbolic set.
\end{enumerate}
\end{lemma}

\subsection{Filtration and structural stability}
To study the stability of the structure of the physical measures, we need the basic properties of filtration, one can see e.g. \cite{Smale} for more details.
Let $f$ be a homeomorphism on compact metric space $M$. 
We say that $\{\mathscr{N}^i\}_{0\le i\le\ell}$ is a \emph{filtration} of $f$ if 
\begin{itemize}
\item 
$
\mathscr{N}^i\subset \mathscr{N}^{i+1}, ~ 1\le i\le \ell-1, \mathscr{N}^{\ell} =M,~\mathscr{N}^0=\emptyset,
$
\smallskip
\item each $\mathscr{N}^i$ is a compact set, and $f(\mathscr{N}^i) \subset {\rm int}(\mathscr{N}^i)$.
\end{itemize}

One can check by definition that for each $1\le i\le \ell$, the set
$$
\Lambda^i(f):=\bigcap_{n\in\ZZ}f^n(\mathscr{N}^i\setminus \mathscr{N}^{i-1})
$$
is a compact invariant subset of $M$.

\begin{lemma}\label{feiyoudang1}
Let $\{\mathscr{N}^i\}_{0\le i\le \ell}$ be a filtration of $f$. Then there exists a \( C^0 \) neighborhood $\U_0$ of \( f \) such that for every homeomorphism \( \tilde{f} \in \U_0\),  \( \{\mathscr{N}^i\}_{0\le i\le \ell} \) forms a filtration of \( \tilde{f} \).
Furthermore, the non-wandering set $\Omega(\tilde{f})$ can be rewrite as a disjoint union of sets:
$$
\Omega(\tilde{f})=\bigcup_{i=1}^{\ell}\Omega^i(\tilde{f}),
$$
where
$$
\Omega^i(\tilde{f})=\Omega(\tilde{f})\cap (\mathscr{N}^i\setminus \mathscr{N}^{i-1})=\Omega(\tilde{f}) \cap \Lambda^i(\tilde{f}).
$$
\end{lemma}

We say that \(\Lambda\) is an {\it isolated hyperbolic set} of \(f\) if there exists a neighborhood \(U\) of \(\Lambda\) such that \(\Lambda = \bigcap_{i \in \mathbb{Z}} f^i(U)\), and the restriction of \(f\) to \(\Lambda\) is uniformly hyperbolic, where \(U\) is called an {\it isolating neighborhood} of \(\Lambda\).

\begin{lemma}\label{feiyoudang2}\cite[Theorem 4.23]{Wen}
Let $\Lambda$ be an isolated hyperbolic set of $f$ and $U$ is an isolating neighborhood of $\Lambda$. Then there exists a \( C^1 \) neighborhood $\U$ of \( f \) such that, for any \( \tilde{f}\in \U \), the maximal invariant set $\Lambda(\tilde{f})$ of $\tilde{f}$ in $U$ is isolated in $U$ and $f|_{\Lambda(\tilde{f})}$ is $C^0$-conjugate to $f|_{\Lambda}$ and the $C^0$ conjugacy  is uniformly close to the identity.
\end{lemma}


\subsection{Proof of Theorem~\ref{yizhi}}

Let us take a $C^{\infty}$ diffeomorphism $f_0=G_k$ for some $k\ge \max\{\widehat{k}_1,\kappa_1\}$, satisfying Lemma \ref{second} and Lemma \ref{DAz}. 
It follows from Lemma \ref{second} that $f_0$ exhibits the partially hyperbolic splitting 
$$
T(\mathbb{T}^5)=E^{u}\oplus_{\succ} E_1^c\oplus_{\succ} E_2^c\oplus_{\succ} E^{s}
$$
with the listed properties there. By the $C^1$-robustness of partial hyperbolicity, it follows that one can find a $C^1$-neighborhood $\U$ of $f_0$ for which any $f\in \U$ admits the partially hyperbolic with the same type, thus $f\in \U$ exhibits properties \eqref{psr} and \eqref{noh}. 

Now we show the  remaining properties of Theorem \ref{yizhi}.
Let $\mathscr{N}^0=\emptyset$ and
$$
\mathscr{N}^1=\TT^2\times \overline{V} \times \left[-\frac{\delta_0}{4},\frac{\delta_0}{4}\right],
$$
$$
\mathscr{N}^2=\TT^2\times\TT^2 \times \left[-\frac{\delta_0}{4},\frac{\delta_0}{4}\right],
$$
$$
\mathscr{N}^3=\TT^2\times\TT^2 \times \left(\left[-\frac{\delta_0}{4},\frac{\delta_0}{4}\right]\cup \left[\frac{1}{2}-\frac{\delta_0}{4}, \frac{1}{2}+\frac{\delta_0}{4}\right]\right),
$$
$$
\mathscr{N}^4= \TT^2\times\TT^2 \times \left[-\frac{\delta_0}{4},\frac{1}{2}+\frac{\delta_0}{4}\right],
$$
$$
\mathscr{N}^5=\TT^2\times \TT^2\times \SS^1=\TT^5.
$$
One knows by construction that $\{\mathscr{N}^i\}_{1\le i \le 5}$ is a family of compact subsets with the nested property: $\mathscr{N}^i\subset \mathscr{N}^{i+1}$ for every $1\le i \le 4$. 
Since $b_1=0$ and $b_3=1/2$ are sinks of $J$, by construction of $f_0$, we conclude that $f_0(\mathscr{N}^i)\subset {\rm int}(\mathscr{N}^i)$ for $2\le i \le 4$. Now we show $f_0(\mathscr{N}^1)\subset {\rm int}(\mathscr{N}^1)$, which together with above demonstrate that $\{\mathscr{N}^i\}_{0\le i \le 5}$ is a filtration of $f_0$. To this end, we observe that 
$$
f_0(x,y,z)=\left(A^2(x),g(y),J(z)\right),\quad \forall~(x,y,z)\in \TT^2\times \TT^2 \times \left[-\frac{\delta_0}{4},\frac{\delta_0}{4}\right],
$$
where $g=A$ outside $V_0=\F_{\delta_0}^{u}(0)\times \F_{\delta_0}^s(0)$, and
$$
g(y)=g(y_1,y_2)=A(y_1,y_2)+\left(\varphi(ky_1)\cdot \varphi(y_2)\cdot \left(\frac{1}{2}-\lambda\right)y_1,0\right),\quad y\in V_0.
$$
By Lemma \ref{DAz}, there is an open neighborhood $V\subset V_0$ such that $g(\overline{V})\subset V$, which together with the contraction of $J$ on $[-\delta_0/4,\delta_0/4]$ ensures that $f_0(\mathscr{N}^1)\subset {\rm int}(\mathscr{N}^1)$.

From properties \eqref{at},\eqref{convg} and \eqref{hyper} in Lemma \ref{DAz}, it follows that
$$
\Lambda^1(f_0)=\TT^2\times\{0\}\times\{0\}
$$
and
$$
\Lambda^2(f_0)=\TT^2\times \left(\TT^2\setminus\bigcup_{n\in \NN}g^{-n}(V)\right)\times\{0\}
$$
are compact invariant hyperbolic set of $f_0$ with unstable indices $1$ and $2$, respectively. Moreover, by the construction of $f_0$, we obtain the following facts:
$$
\Lambda^3(f_0)=\TT^2\times\{1/2\}\times\TT^2, 
$$
$$
\Lambda^4(f_0)=\TT^2\times\{1/4\}\times\TT^2,
$$
$$
\Lambda^5(f_0)=\TT^2\times\{3/4\}\times\TT^2.
$$
they are compact invariant hyperbolic subsets of $f_0$ with unstable indices $2,3,3$, respectively. Observe also 
$$
{\rm Leb}\left(\Lambda^i(f_0)\right)=0, \quad 1\le i \le 5.
$$

For each $1\le i \le 5$, we have the fact 
$$
\Lambda^i(f_0):=\bigcap_{n\in\ZZ}f_0^n\left(\mathscr{N}^i\setminus \mathscr{N}^{i-1}\right)=\bigcap_{n\in\ZZ}f_0^n\left({\rm int }(\mathscr{N}^i)\setminus \mathscr{N}^{i-1}\right).
$$
Hence, $\Lambda^i(f_0)$ is an isolated hyperbolic set of $f_0$ and ${\rm int }(\mathscr{N}^i)\setminus \mathscr{N}^{i-1}$ is an isolating neighborhood of $\Lambda^i(f_0)$. As a result of Lemma~\ref{feiyoudang2}, up to shirinking $\U$, we get that for every $f\in \U$, $\Lambda^i(f)$ is a compact invariant hyperbolic subset of $f$, with the same unstable index to $\Lambda^i(f_0)$, for every $1\le i \le 5$. Moreover, as 
$$
\Omega(f)\subset \bigcup_{i=1}^5\Lambda^i(f)
$$
by Lemma \ref{feiyoudang1}, every invariant measure is supported on $\bigcup_{i=1}^5\Lambda^i(f)$. This implies that every invariant measure of $f\in \U$ is hyperbolic, verifying \eqref{hyu} of Theorem \ref{yizhi}.

\medskip

To show the remaining properties of Theorem \ref{yizhi}, we need the following observation.  

\begin{lemma}\label{tupi}
Shrinking $\U$ if necessary, for any $C^{1+}$ diffeomorphism $f\in \U$, there exists a subset $Z_f\subset \mathscr{N}^2$ with zero Lebesgue measure such that, $\omega(x,f)\subset\Lambda^1(f)$ for any $x\in \mathscr{N}^2\setminus Z_f$.
\end{lemma}

We assume this for a while, and show how to complete the proof of this theorem from it. 
Define
$$
\K=\TT^2\times \TT^2\times \left[\frac{1}{4}-\delta_0,\frac{1}{4}+\delta_0\right],
$$
then $f_0^{-1}(\K)\subset {\rm int}(\K)$. Note also that $\Lambda^4(f_0)$ is a repeller that can be rewritten as
$$
\Lambda^4(f_0)=\bigcap_{n\in \NN}f_0^{-n}(\K).
$$
Because $f_0$ coincides with $A^2\times A \times J$ on $\K$, which implies
$$
|{\rm det} Df^{-1}_0(x)|=\|DJ^{-1}(x)\|\le \frac{3}{2\lambda}<1/4<1, \quad \forall~ x\in \K.
$$
Up to shrinking $\U$, we assume that for every $f\in \U$, it holds that 
\begin{itemize}
\item $f^{-1}(\K)\subset {\rm int}(\K)$ and $|{\rm det} Df^{-1}(x)|<1, \quad \forall~ x\in \K.$
\smallskip
\item $\Lambda^4(f)=\bigcap_{n\in \NN}f^{-n}(\K).$
\end{itemize}
As a result, ${\rm Leb}(\Lambda^4(f))=0$. With the same argument, one also gets 
$
{\rm Leb}(\Lambda^5(f))=0
$
for every $f\in \U$, reducing $\U$ if necessary.

\medskip
For every fixed $C^{1+}$ diffeomorphism $f\in \U$, we take 
$$
\Gamma_f=\TT^5\setminus \left(\Lambda^4(f)\cup\Lambda^5(f)\cup \left(\bigcup_{n\in \NN}f^{-n}(Z_f)\right)\right).
$$
By above result and Lemma \ref{tupi}, one knows that $\Gamma_f$ has full Lebesgue measure in $\TT^5$. 
Denote the forward limit set of a point $x$ under  $f$ by 
$$
\omega(x, f) = \Big\{ y \in \mathbb{T}^5 : \exists\, \{n_i\} \subset \mathbb{N} \text{ such that } \lim_{i \to +\infty} f^{n_i}(x) = y \Big\}.
$$
We claim that:

\begin{claim}\label{in2}
For any $x\in \Gamma_f$, either $\omega(x,f)\subset \Lambda^1(f)$ or $\omega(x,f)\subset \Lambda^3(f)$.
\end{claim}

\begin{proof}
Fix any $x\in \Gamma_f$.
Since $x\notin \Lambda^5(f)$, while
$$
\Lambda^5(f)=\bigcap_{n\in \NN}f^{-n}(\TT^5\setminus \mathscr{N}^4),
$$
there is $n_1\in \NN$ so that $x_1:=f^{n_1}(x)\in \mathscr{N}^4$. Note that 
$$
x_1\notin \Lambda^4(f)=\bigcap_{n\in \NN}f^{-n}(\mathscr{N}^4\setminus \mathscr{N}^3)
$$
as $x\notin \Lambda^4(f)$, there exists $n_2\in \NN$ such that 
$x_2:=f^{n_2}(x_1)\notin \mathscr{N}^4\setminus \mathscr{N}^3$. 
Observe that $x_2\in \mathscr{N}^4$ as $x_1\in \mathscr{N}^4$ and $f(\mathscr{N}^4)\subset {\rm int}(\mathscr{N}^4)$, thus $x_2\in \mathscr{N}^3$.
If $x_2\in \mathscr{N}^3\setminus \mathscr{N}^2$, as $\Lambda^3(f)$ is an attractor with the trapping region $\mathscr{N}^3\setminus \mathscr{N}^2$, we conclude that $\omega(x,f)=\omega(x_2,f)\subset \Lambda^3(f)$.
Otherwise, we have $x_2\in \mathscr{N}^2\setminus Z_f$ as $x_2=f^{n_1+n_2}(x)\notin Z_f$ by construction of $\Gamma_f$. By applying Lemma \ref{tupi}, $\omega(x,f)\subset \Lambda^1(f)$.
\end{proof}

For any $C^{1+}$ diffeomorphism $f$ in $\U$, denote
$$
B_1=\{x\in \Gamma_f: \omega(x,f)\subset \Lambda^1(f)\},\quad
B_2=\{x\in \Gamma_f: \omega(x,f)\subset \Lambda^3(f)\}.
$$
According to Claim \ref{in2}, the union of $B_1$ and $B_2$ exhibits full Lebesgue measure in $\TT^5$. Observe that $\mathscr{N}^2\setminus Z_f\subset B_1$ and $\mathscr{N}^3\setminus \mathscr{N}^2\subset B_2$, we know ${\rm Leb}(B_1)>0$ and ${\rm Leb}(B_2)>0$.
When $x\in B_1$, then any measure $\mu\in \mathscr{M}(x,f)$ is supported on $\Lambda^1(f)$. 
Since $f|_{\Lambda^1(f)}$ is topological conjugate to $f_0|_{\Lambda^1(f_0)}$, and $f_0$ is uniformly contracting along $E_1^c\oplus E_2^c$, $f$ is uniformly contracting along $E_1^c\oplus E_2^c$ as well. Consequently,
$$
\chi(z, E_1^c)<0,\quad \mu-\textrm{a.e.}~z.
$$
When $x\in B_2$, as $f_0$ is uniformly expanding along $E_1^c$ and uniformly contracting along $E_2^c$, it follows that for every $\mu\in \mathscr{M}(x,f)$, it holds that 
$$
\chi(z, E_1^c)>0>\chi(z,E_{2}^c),\quad \mu-\textrm{a.e.}~z.
$$
This completes the proof of Item \eqref{twofen}. 

\medskip
By \eqref{hyu} and \eqref{twofen}, using the result \cite[Theorem B]{MC25}(or \cite[Theorem 2]{Bur25}), we know that any $C^{1+}$ diffeomorphism $f\in \U$ admits finitely many hyperbolic ergodic physical measures, whose basins cover a set of full Lebesgue measure in $\TT^5$. 
Moreover, by the above claim these physical measures are supported either on $\Lambda^1(f)$ or on $\Lambda^3(f)$. In particular, $f_0$ has a unique physical measure ${\rm Leb}_{\TT^2\times \{0\}\times \{0\}}$ supported on $\Lambda^1(f_0)$, and a unique physical measure ${\rm Leb}_{\TT^2\times \TT^2 \times \{1/2\}}$ supported on $\Lambda^3(f_0)$.
By Lemma~\ref{feiyoudang2},  $f$ is topologically mixing in $\Lambda^1(f)$ and $\Lambda^3(f)$, just as \( f_0 \) is. It follows that there exists a unique physical measure $\mu_1$ supported on $\Lambda^1(f)$, and a unique physical measure $\mu_2$ supported on $\Lambda^3(f)$, with indices $1$ and $2$ respectively. 
This completes the proof of Item \eqref{ptw} of Theorem \ref{yizhi}.

Hence, we are left to show Lemma \ref{tupi}.

\begin{proof}[Proof of Lemma \ref{tupi}]
Write $R_g$ for
$
\TT^2\setminus\bigcup_{n\in \NN}g^{-n}(V).
$
Note first that $f_0=A^2\times g\times J$ on $\TT^2\times \TT^2\times [-\delta_0/4,\delta_0/4]$, which is uniformly hyperbolic on 
$$
\Lambda^2(f_0)=\TT^2\times R_g\times\{0\}.
$$
In particular, by Lemma \ref{DAz}, $g$ is uniformly hyperbolic on $R_g$, which is uniformly expanding along $E^{cu}=\cE^u$. Notice that as $E_1^c=\cE^u$ provided in Lemma \ref{second}, we have $E^{cu}=E_1^c=\cE^u$. 
Let \( W^u(x, g) \) denote the unstable manifold of \( x \in R_g\) w.r.t. \( g \), which is tangent to $\cE^u$ and contains its respective local unstable manifolds of a uniform size \( \varepsilon_0 \).
Then, the global unstable manifold of \( p = (x, y, 0) \in \Lambda^2(f_0) \), denoted by \( W^u(p, f_0) \), has the following form:  
\[
W^u(p, f_0) = \mathcal{F}^{uu}(x, A^2) \times W^u(y, g) \times \{0\},
\]  
where \( \mathcal{F}^{uu}(x, A^2) \)  denotes the unstable manifold of \( x \in \mathbb{T}^2 \) for the Anosov map \( A^2 \).

Let
$$
\mathcal{B}=\TT^2\times V \times \left(-\frac{\delta_0}{4}, \frac{\delta_0}{4}\right).
$$
We show now
\begin{equation}\label{inter}
W^u(p,f_0)\cap \mathcal{B}\neq \emptyset,\quad \forall~p\in \Lambda^2(f_0).
\end{equation}
This comes from the following fact:
$W^u(y,g)\cap V\neq \emptyset$ for every $y\in R_g$. Arguing by contradiction, assume there is $y_0\in R_g$ for which 
\begin{equation}\label{cap0}
W^u(y_0,g)\cap V=\emptyset.
\end{equation}
Define
$$
H_m=\TT^2\setminus\bigcup_{n=0}^{m}g^{-n}(V),\quad m\in \NN.
$$
We check by definition that each $H_m$ is a closed neighborhood of $R_g$, and satisfies $g^{-1}(H_m)=H_{m+1}\subset {\rm int}(H_m)$. Moreover, one can choose $N$ large enough such that $g$ keeps uniformly expanding along $E_1^c$ on $H_{N}$. Property \eqref{cap0} implies that $W^u(g^{-N}(y_0),g)\cap g^{-N}(V)=\emptyset$, together with the obervation that $H_N=\TT^2\setminus g^{-N}(V)$, this yields
\begin{equation}\label{hn}
W^u(y_N,g)\subset H_{N}, \quad \textrm{where}~y_N:=g^{-N}(y_0).
\end{equation}
As a result, 
$$
W^u(g^{-m}(y_N),g)=g^{-m}(W^u(y_N,g))\subset g^{-m}(H_N)\subset H_N, \quad \forall~m\in \NN.
$$
On the other hand, as the unstable manifold $\F^u(y_N,A)$ of Anosov map $A$ is dense in $\TT^2$, there is $S>0$ such that $\F_S^u(y_N,A)\cap H_N^c\neq \emptyset$. Using the fact $E_1^c=\cE^u$ and the uniform expansion of $g$ on $H_N$, we can find $m_0\in \NN$ to satisfy
$$
W^u(y_N,g)\supset g^{m_0}\left(W^u_{\varepsilon}(g^{-m_0}(y_N),g)\right)\supset \F_S^u(y_N,A).
$$
Thus, $W^u(y_N,g)\cap H_N^c\neq \emptyset$, which contradicts \eqref{hn} and we get \eqref{inter}.

By \eqref{inter} and the continuity of unstable manifolds w.r.t. points and diffeomorphisms, for every $p\in \Lambda^2(f_0)$, there exists an open neighborhood $W_p$ of $p$, $L_p>0$ and a $C^1$ open neighborhood $\U_p\subset \U$ of $f_0$ so that 
$$
W^u_{L_p}(q,f) \cap \mathcal{B} \neq \emptyset,\quad \forall~q\in W_p\cap \Lambda^2(f), \quad \forall~f\in \U_p.
$$
By compactness of $\Lambda^2(f_0)$, one can find an open covering $\{W_{p_i}: 1\le i \le m\}$ of $ \Lambda^2(f_0)$. Take $\V$ so that $\Lambda^2(f)\subset \cup_{1\le i \le m}W_{p_i}$ for every $f\in \V$. Hence, by taking $\widehat{\U}=\V \cap \U_{p_1}\cap \cdots \cap\U_{p_m}$, one has that
\begin{equation}\label{inout}
W^u(p,f) \cap \mathcal{B} \neq \emptyset,\quad \forall~p\in \Lambda^2(f), \quad \forall~f\in \widehat{\U}.
\end{equation}
To conclude, one gets that for every $f\in \widehat{\U}$, for every point in $\Lambda^2(f)$, its unstable manifold can not entirely contained in $\Lambda^2(f)$.

\medskip
Given $f\in \widehat{\U}$ and $x\in \mathscr{N}^2$. Since $\mathscr{N}^2$ is positively invariant under $f$ and the forward limit set is contained in the non-wandering set, we know $\omega(x,f)\subset \Omega(f)\cap \mathscr{N}^2$. Using Lemma \ref{feiyoudang1}, we obtain
$\omega(x,f)\subset \Lambda^1(f)\cup \Lambda^2(f)$. Observe that $\omega(x,f)$ cannot be written as a union of disjoint compact invariant sets, it follows that
\begin{itemize}
\item either, $\omega(x,f)\subset\Lambda^1(f)$,
\smallskip
\item or, $\omega(x,f)\subset\Lambda^2(f)$.
\end{itemize}
Therefore, to conclude the result, it suffices to show the following set
$$
Z_f:=\{x\in \mathscr{N}^2: \omega(x,f)\subset \Lambda^2(f)\}
$$
admits zero Lebesgue measure. 

Assume by contradiction that ${\rm Leb}(Z_f)>0$.
Since $f$ is uniformly expanding along $E^{cu}=E^u\oplus E_1^c$ for points in $\Lambda^2(f)$, there exists a constant \( \alpha > 0 \) such that  
\[
\|Df^{-1}|_{E^{cu}(y)}\| \le e^{-\alpha}<1,
\quad \forall~y\in \Lambda^2(f).
\]  
For any \( x \in Z_f \), since \( \lim_{n \to +\infty} d(f^n(x), \omega(x,f)) = 0 \) and \( \omega(x,f) \subset \Lambda^2(f) \), using the uniform continuity of $\mathscr{N}^2\ni y\mapsto \log \|Df^{-1}|_{E^{cu}(y)}\|$, one gets
\begin{equation}\label{nuab}
\limsup_{n \to +\infty} \frac{1}{n} \sum_{i=1}^{n} \log \|Df^{-1}|_{E^{cu}(f^i(x))}\| \le -\alpha<0.
\end{equation}
This means that every point in the positive Lebesgue measure subset $Z_f$ is non-uniformly expanding along $E^{cu}$. 
Noting also that $Z_f$ is contained in the compact and positively invariant subset $\mathscr{N}^2$. 
By applying \cite[Theorem A]{ABV00}, one can find an ergodic hyperbolic SRB measure $\nu$ for $f$ with the following properties:
\begin{itemize}
\item $Z_f\cap B(\nu,f)$ admits positive Lebesgue measure,
\smallskip
\item the conditional measures of $\nu$ along Pesin unstable manifolds tangent to \( E^{cu} \) are equivalent to the Lebesgue measure.
\end{itemize} 
The first property guarantees that one can take a point $x_0$ in $Z_f\cap B(\nu,f)$. Observe first that $x_0\in B(\nu,f)$ implies ${\rm supp}(\nu)\subset \omega(x_0,f)$. On the other hand, from $x_0\in Z_f$ and the definition of $Z_f$, we have $\omega(x_0,f)\subset \Lambda^2(f)$. Therefore, ${\rm supp}(\nu)\subset \Lambda^2(f)$.
The second property implies that \( \nu \) is \( u \)-saturated; that is, almost every unstable manifold of dimension \( 2 \) is contained in the support of \( \nu \)(see e.g. \cite[Corollary 2.1]{cv}). As a consequence,
there exists an unstable manifold \( W^u(p, f) \) entirely contained in the support of \( \mu \) for some $p\in \Lambda^2(f)$, and hence included in \( \Lambda^2(f) \). This contradicts \eqref{inout}, since \( \Lambda^2(f) \subset \mathcal{B} \).

\end{proof}

\section{Complete the proof of Theorem \ref{maintwo}: the general case}\label{mg}
A detailed construction of case $m=2$ has been provided in Section \ref{twop}. In what follows, first we show how to construct the examples of case $m=3$ from examples of the case $m=2$, the general case then follows by induction, applying the same method.

\begin{proof}[Complete the proof of Theorem \ref{maintwo}]
By Theorem \ref{yizhi}, in the special case $m=2$, we have constructed the partially hyperbolic system $f_0=G_k\in {\rm Diff}^{\infty}(\TT^5)$, which admits two hyperbolic ergodic physical measures with unstable indices $1$ and $2$. According to its construction, one knows that $f_0$ can be viewed as the product of the system \( (A^2, \mathbb{T}^2) \) and the skew-product system $(I, \SS^1\times \TT^2)$ defined as
$$
I(x,y)=\left(J(x), A_x(y)\right), \quad (x,y)\in \SS^1\times \TT^2,
$$
where $J$ is the Morse-Smale diffeomorphism on $\SS^1$ given in Subsection \ref{Gk}. 

 Notice that the product system $(A^3\times A^2 \times I, \TT^2\times \TT^2\times \SS^1\times \TT^2)$ admits two physical measures with unstable indices $2$ and $3$.  As a first step, we modify $A^3 \times A^2 \times I$ on $\mathbb{T}^2 \times \mathbb{T}^2 \times \mathbb{T}^2 \times [-\delta_0, \delta_0]$ so that the resulting system admits only one physical measure on $\mathbb{T}^2 \times \mathbb{T}^2 \times \mathbb{T}^2 \times [-\delta_0, \delta_0]$, whose unstable index is one.
 To this end, we present the following construction to realize this.



Consider the product map \( g_k \times A^2 \) on \( \mathbb{T}^2 \times \mathbb{T}^2 \), where \( g_k : \mathbb{T}^2 \to \mathbb{T}^2 \) is a diffeomorphism given in Lemma \ref{DAz}.
One can check by definition that on the rectangle 
$$ \left[ -\dfrac{\delta_0}{4k}, \dfrac{\delta_0}{4k} \right] \times \left[ -\dfrac{\delta_0}{4}, \dfrac{\delta_0}{4} \right],
$$
the map \( g_k(u, v) \) coincides with the linear transformation  
\[
g_k(u, v) = \left( \frac{1}{2}u, \frac{1}{\lambda}v \right).
\]
Thus, we may assume that on the rectangle  
\[
\left[ -\frac{\delta_0}{4k}, \frac{\delta_0}{4k} \right]^4,
\]  
the map \( g_k \times A^2 \) coincides with the linear transformation  
\[
\left(g_k \times A^2\right)(x_1, x_2, x_3, x_4) = \left( \frac{1}{2}x_1, \frac{1}{\lambda}x_2, \lambda^2 x_3, \frac{1}{\lambda^2}x_4 \right).
\]
We begin by modifying $g_k \times A^2$ to construct a sequence of $C^{\infty}$ diffeomorphisms. Let us construct $H_k$ in the following way:
\begin{itemize}
\smallskip
\item $H_k$ agrees with $g_k \times A^2$ for points outside $\left(-\dfrac{\delta_0}{4k}, \dfrac{\delta_0}{4k} \right)^4$, 
\smallskip
\item if $x=(x_1,x_2,x_3,x_4)\in \left(-\dfrac{\delta_0}{4k}, \dfrac{\delta_0}{4k} \right)^4$,
put
$$
H_k(x)=\left( \frac{1}{2}x_1, \frac{1}{\lambda}x_2, R_k(x)+\lambda^2 x_3, \frac{1}{\lambda^2}x_4 \right),
$$
where 
$$
R_k(x) = \varphi(k^2x_3)\cdot \varphi\left(k^2\sqrt{x_1^2+x_2^2+x_4^2}\right) \cdot\left( \frac{3}{4}- \lambda^2\right)x_3.
$$
\end{itemize}

Observe that \( R_k(x)+\lambda^2 x_3\) differs from the original map in the domain 
$$ 
\left[ -\frac{\delta_0}{2k^2}, \frac{\delta_0}{2k^2} \right]^4,
$$ 
where the derivative of the map with respect to \( x_3 \) gradually changes from a minimum value of \( 3/4 \) to $\lambda_2$.
With the argument as in Lemma \ref{first}, one sees that for sufficiently large \( k \), $H_k$ admits a dominated splitting into four one-dimensional invariant sub-bundles. Thus, on the invariant set \( \{0\} \times \mathbb{T}^2 \), the map \( H_k \) resembles the system \( (\mathbb{T}^2, g_k) \), exhibiting both a repeller and a fixed sink.  Furthermore, on the set $\left[ -\frac{\delta_0}{2k^2}, \frac{\delta_0}{2k^2} \right]^2 \times \mathbb{T}^2$, the map $H_k$ is  similar to the system $([-\frac{1}{4}\delta_0,\frac{1}{4}\delta_0] \times \mathbb{T}^2, J \times g_k)$.  Building upon this structure, we now define the map via the following construction on $\TT^2\times\TT^2\times[-\delta_0,\delta_0]$:
$x=(x_1,x_2,x_3,x_4,x_5)\in \TT^2\times\TT^2\times[-\delta_0,\delta_0]$,
put
$$
A_k(x)=\left(\varphi(x_5)Q_k(x)+\lambda x_1,\lambda^{-1}x_2,\varphi(2x_5)\widehat{R}_k(x)+\lambda^2 x_3,\lambda^{-2}x_4,J(x_5)\right),
$$
where 
$$
Q_k(x) = \varphi(kx_1)\cdot \varphi(x_2) \cdot\left( \frac{1}{2}- \lambda\right)x_1,
$$
and 
$$
\widehat{R}_k(x)=R_k(x_1,x_2,x_3,x_4).
$$
Thus, it is easy to verify that when \( x \in \mathbb{T}^2 \times \mathbb{T}^2 \times \left[ -\dfrac{\delta_0}{8}, \dfrac{\delta_0}{8} \right] \), we have  
\[
A_k = H_k \times J.
\]
Hence, we can verify that the product map \( A^3 \times A_k \) admits a physical measure whose support is $\mathbb{T}^2 \times \{0\}$ and whose unstable index is equal to one.

We next construct a system on $\mathbb{T}^2 \times \mathbb{T}^2 \times \mathbb{T}^2 \times \mathbb{S}^1$ that admits physical measures with three distinct unstable indices.
Observe that the map $A^3 \times A^2 \times I$, defined on $\mathbb{T}^2 \times \mathbb{T}^2 \times \mathbb{T}^2 \times [-\delta_0, \delta_0]$, is modified to $A^3 \times A_k$, with the modification supported strictly inside $\mathbb{T}^2 \times \mathbb{T}^2 \times \mathbb{T}^2 \times \left( -\frac{1}{4} + 2\delta_0, \frac{1}{4} - 2\delta_0 \right)$. We then modify $J$ to obtain $J_6$ by adding a sink $a$ and a source $b$ to the circle, satisfying properties (J\ref{j1}), (J\ref{j2}), and (J\ref{j3}). Consequently, $J_6$ has three sinks $a_1, a_2, a_3$ and three sources $b_1, b_2, b_3$,  arranged in counterclockwise order on the circle as $a_1, b_1, a_2, b_2, a_3, b_3$. Following the counterclockwise orientation of the circle, we construct a partially hyperbolic diffeomorphism $f_0 \in \mathrm{Diff}^{\infty}(\mathbb{T}^7)$ such that 
$$
f_0 = A^3 \times A_k \quad \text{on } \mathbb{T}^2 \times \mathbb{T}^2 \times \mathbb{T}^2 \times [b_3, b_1],
$$
where the point $a_1$ is identified with $0 \in S^1$ in the $A^3 \times A_k$ component,
and
$$
f_0 = A^3 \times A^2 \times I \quad \text{on } \mathbb{T}^2 \times \mathbb{T}^2 \times \mathbb{T}^2 \times [b_1, b_3],
$$
where the points $a_2$ and $a_3$ are identified with $0$ and $\frac{1}{2} \in S^1$, respectively, in the $A^3 \times A^2 \times I$ component.
Since $A^3 \times A^2 \times I$ and $A^3 \times A_k$ coincide in a neighborhood of each slice $\mathbb{T}^2 \times \mathbb{T}^2 \times \mathbb{T}^2 \times \{b_i\}$, the map $f_0$ is well-defined and smooth, where each $b_i$ is a source of $J_6$.
Following the same argument in the proof of Theorem \ref{yizhi}, one can obtain an $C^1$ open neighborhood $\U$ of $f_0$ demonstrating Theorem \ref{maintwo} for $m=3$. By inductively applying this approach, we establish Theorem \ref{maintwo} in full generality.
\end{proof}

\begin{remark}
For the example of physical measures with an arbitrary number of distinct unstable indices, we adopted a recursive construction method, unraveling a more intricate structure with each step. The fundamental strategy mirrors that used for the case with two distinct unstable indices. During the recursive process, the corresponding elements from previous steps can be recovered and maintained. As a result, physical measures with the given number of distinct unstable indices naturally emerge.
\end{remark}

\appendix
\section{Proof of Proposition \ref{DAz}}\label{app}

\begin{proof}[Proof of Proposition \ref{DAz}]
By Lemma~\ref{second}, \( \kappa_{\varepsilon} \) can be chosen to be \( \widehat{k}_{\varepsilon} \), which satisfies the conditions for the partially hyperbolic requirements.
Fix any \( k\ge\kappa_{\varepsilon}  \), let
$$
L(u,v)=\varphi(ku)\cdot \varphi(v)\cdot \left(\frac{1}{2}-\lambda\right)u+\lambda u,\quad \forall~ (u,v)\in  V_0.
$$
Thus,
$$
g(u,v)=(L(u,v),\lambda^{-1}v), \quad \forall~(u,v)\in V_0.
$$
On $V_0$ we have the expression  
\begin{equation*}\label{exp}
\renewcommand{\arraystretch}{1.3}
Dg(u,v)=
\begin{pmatrix}
\frac{\partial }{\partial u}L(u,v) & \frac{\partial }{\partial v}L(u,v)\\
0 & \lambda^{-1}
\end{pmatrix},
\end{equation*}
where 
$$
\frac{\partial }{\partial u}L(u,v)=\lambda+\left(\frac{1}{2}-\lambda\right)\cdot \varphi(v)\cdot(ku\varphi'(ku)+\varphi(ku)),
$$
and 
$$
\frac{\partial }{\partial v}L(u,v)=\varphi(ku)\varphi'(v)\left(\frac{1}{2}-\lambda\right)\cdot u.
$$
By a simple computation, one knows that on $V_0$ there exists a sink $(0,0)$ and two saddle points 
$(\pm u_0, 0)$ within
$$
\varphi(ku_0)=\frac{\lambda-1}{\lambda-\frac{1}{2}}, \quad \frac{\delta_0}{4k}<u_0<\frac{\delta_0}{2k}.
$$

Since 
$$
\frac{\partial }{\partial u}L(u_0,0)=\lambda+\left(\frac{1}{2}-\lambda\right)\cdot(ku_0\varphi'(ku_0)+\varphi(ku_0))>1,
$$
we can also assume that there is a neighborhood $B_r(u_0)$ of $u_0$ such that for each $x\in B_r(u_0)$, $v\in[-\frac{\delta_0}{2},\frac{\delta_0}{2}]$
\begin{align*}
\begin{split}
\frac{\partial }{\partial u}L(x,v)&=\lambda+\left(\frac{1}{2}-\lambda\right)\cdot \varphi(v)\cdot(kx\varphi'(kx)+\varphi(kx))\\
&\ge\lambda+\left(\frac{1}{2}-\lambda\right)\cdot(kx\varphi'(kx)+\varphi(kx))\\
&>1.
\end{split}
\end{align*}
(Without loss of generality, we can assume \( kx\varphi'(kx)+\varphi(kx) > 0 \), since when \( kx\varphi'(kx)+\varphi(kx) \leq 0 \), the inequality greater than 1 follows easily.)

Take a constant \( a \in (0, u_0) \cap B_r(u_0) \), then  
\[
\frac{\partial}{\partial u} L(a, v) > 1
\]
for all \( v \in \left( -\frac{\delta_0}{4}, \frac{\delta_0}{4} \right) \).  
Define the set  
\[
V := (-a, a) \times \left( -\frac{\delta_0}{4}, \frac{\delta_0}{4} \right).
\]
Since the map \( L(\cdot, \frac{\delta_0}{4}) : [-u_0, u_0] \to [-u_0, u_0] \) is a Morse-Smale diffeomorphism, it follows that \( g(\overline{V}) \subset V \). Moreover, for every point \( y \in V \), we have  
\[
\lim_{n \to +\infty} g^n(y) = (0,0).
\]

We can check that for $x>u_0$,  $\varphi(kx)<\frac{1-\lambda}{\frac{1}{2}-\lambda}$. It follows that for any $v\in[-\frac{\delta_0}{2},\frac{\delta_0}{2}]$, we have
\begin{align*}
\begin{split}
\frac{\partial }{\partial u}L(x,v)&=\lambda+\left(\frac{1}{2}-\lambda\right)\cdot \varphi(v)\cdot(kx\varphi'(kx)+\varphi(kx))\\
&\ge\lambda+\left(\frac{1}{2}-\lambda\right)\cdot(kx\varphi'(kx)+\varphi(kx))( \text{assume \( kx\varphi'(kx)+\varphi(kx) > 0\)})\\
&\ge \lambda+\left(\frac{1}{2}-\lambda\right)\cdot(\varphi(kx))\text{~~(recall $kx\varphi'(kx)\le0$)}\\
&>1. \text{~~(recall $\varphi(kx)<\frac{1-\lambda}{\frac{1}{2}-\lambda}$)}
\end{split}
\end{align*}
It is easy to obtain that when $|v|\ge \delta_0/2$, $\dfrac{\partial }{\partial u}L(x,v)=\lambda>1$. Thus, we have shown that, outside the region $[-a, a] \times \left[ -\delta_0/2, \delta_0/2 \right]$,  $\dfrac{\partial L}{\partial u}$ is strictly greater than 1.

 Let 
$$
R_y\times\{y\}:=\{(x,y)|\frac{\partial }{\partial u}L(x,y)\le1\}.
$$
$$
L:=\bigcup_{y\in\left[ -\frac{\delta_0}{2}, \frac{\delta_0}{2} \right]}R_y\times\{y\}.
$$
From the analysis above, the set \( L \) contains all points in the entire space \( \mathbb{T}^2 \) for which \( \frac{\partial}{\partial u} L(x, y) \le 1 \). Thus,  we conclude that the set  
\[
L \subset [-a, a] \times \left[ -\frac{\delta_0}{2}, \frac{\delta_0}{2} \right].
\]  
Combined with the fact that $g(\overline{V}) \subset V$, to prove that
$$
\mathbb{T}^2 \setminus \bigcup_{i=0}^{+\infty} g^{-i}(V)
$$
forms a compact, invariant, hyperbolic set, it suffices to show that $g(L) \subset \overline{V}$.

Denote by \( \mathrm{Leb} \) the Lebesgue measure on space \( [-a, a] \). We can identify \( R_y \times \{y\} \) with \( R_y \). Thus, for a well-defined set \( R_y \) (which is at least a measurable set), we have 
\[
\Leb(g(R_y)) = \int_{R_y} \frac{\partial }{\partial u}L(x,y) \, {\rm d}\Leb \le \Leb(R_y).
\]

By the symmetry of \( \phi(\cdot) \) and the fact that \( g \) maps each leaf tangent to the unstable direction \( \cE^u \) into another leaf also tangent to \( \cE^u \), we can conclude that within the set \( L \subset [-a, a] \times \left[-\delta_0/2, \delta_0/2\right] \), the set \( R_y \) is symmetric about the line \( x = 0 \). Consequently, the image of \( R_y \) under the map \( g \) remains symmetric about \( x = 0 \), and its Lebesgue measure does not exceed \( 2a \). Furthermore, since \( g \) uniformly contracts all points in the interval \( \left[-\delta_0/2, \delta_0/2\right] \) toward the origin, it follows that 
\[
g(L) \subset \overline{V}=[-a, a] \times \left[-\frac{\delta_0}{4}, \frac{\delta_0}{4}\right].
\]  
\end{proof}

\section{Conflict of Interest and Data Availability}
The authors declare that there are no conflicts of interest in the publication of this paper. No data was generated or analyzed during the study.

\vskip 5pt

\flushleft{\bf Zeya Mi} \\
\small School of Mathematics and Statistics, 
Nanjing University of Information Science and Technology, Nanjing, 210044, P.R. China\\
\textit{E-mail:} \texttt{mizeya@163.com}\\
\flushleft{\bf Hangyue Zhang} \\
\small School of Mathematical Sciences,  Nanjing University, Nanjing, 210093, P.R. China\\
\textit{E-mail:} \texttt{zhanghangyue@nju.edu.cn}\\
\end{sloppypar}
\end{document}